\documentclass{amsart}

\usepackage{blkarray}
\usepackage{amssymb,amsmath,amsrefs}
\usepackage{wrapfig}
\usepackage{subcaption}
\usepackage{enumerate}
\usepackage{wasysym}  
\usepackage{stmaryrd}
\usepackage{listings}
\usepackage{verbatim}
\usepackage[dvipsnames]{xcolor}
\usepackage{pgfplots}
\usepackage{tikz}
\usepackage{tikz-cd}
\pgfplotsset{compat=1.17}
\usetikzlibrary{positioning}
\usepackage{setspace}
\usepackage{hyperref}
\hypersetup{
    colorlinks = false,
    linkbordercolor = {white},
}

\makeatletter
\DeclareFontFamily{OMX}{MnSymbolE}{}
\DeclareSymbolFont{MnLargeSymbols}{OMX}{MnSymbolE}{m}{n}
\SetSymbolFont{MnLargeSymbols}{bold}{OMX}{MnSymbolE}{b}{n}
\DeclareFontShape{OMX}{MnSymbolE}{m}{n}{
    <-6>  MnSymbolE5
   <6-7>  MnSymbolE6
   <7-8>  MnSymbolE7
   <8-9>  MnSymbolE8
   <9-10> MnSymbolE9
  <10-12> MnSymbolE10
  <12->   MnSymbolE12
}{}
\DeclareFontShape{OMX}{MnSymbolE}{b}{n}{
    <-6>  MnSymbolE-Bold5
   <6-7>  MnSymbolE-Bold6
   <7-8>  MnSymbolE-Bold7
   <8-9>  MnSymbolE-Bold8
   <9-10> MnSymbolE-Bold9
  <10-12> MnSymbolE-Bold10
  <12->   MnSymbolE-Bold12
}{}

\let\llangle\@undefined
\let\rrangle\@undefined
\DeclareMathDelimiter{\llangle}{\mathopen}%
                     {MnLargeSymbols}{'164}{MnLargeSymbols}{'164}
\DeclareMathDelimiter{\rrangle}{\mathclose}%
                     {MnLargeSymbols}{'171}{MnLargeSymbols}{'171}
\makeatother

\newtheorem{thm}{Theorem}[section]
\newtheorem{coro}[thm]{Corollary}
\newtheorem{defn}[thm]{Definition}
\newtheorem{prop}[thm]{Proposition}

\theoremstyle{remark}
\newtheorem{ex}[thm]{Example}
\newtheorem{rem}[thm]{Remark}

\newcommand{\K}{\mathbb{K}}
\newcommand{\R}{\mathbb{R}}
\newcommand{\C}{\mathbb{C}}

\renewcommand{\leq}{\leqslant} 
\renewcommand{\geq}{\geqslant}

\newcommand{\ch}{\mathcal{C}}

\DeclareMathOperator{\Cen}{Cen}

\DeclareMathOperator{\Der}{Der}

\usepackage{xcolor}

\definecolor{cerulean}{rgb}{0,.48,.65} 
\definecolor{magenta}{rgb}{.5,0,.5} 
\definecolor{dred}{rgb}{.5,0,0} 
\definecolor{green}{rgb}{0,.5,0} 
\definecolor{blue}{rgb}{0,0,0.5} 
\definecolor{black}{rgb}{0,0,0} 
\definecolor{dgreen}{rgb}{0,.3,0} 
\definecolor{vdred}{rgb}{.3,0,0} 
\definecolor{red}{rgb}{1,0,0} 
\definecolor{salmon}{rgb}{0.98,0.50,0.45} 
\definecolor{gray}{rgb}{.5,.5,.5} 
\definecolor{seagreen}{rgb}{0.13,0.70,0.67} 
\definecolor{chartreuse}{rgb}{0.40,0.80,0.00}
\definecolor{cornflower}{rgb}{0.39,0.58,0.93} 
\definecolor{gold}{rgb}{0.80,0.68,0.00}
\definecolor{cornellred}{rgb}{0.7, 0.11, 0.11} 
\definecolor{csugreen}{rgb}{0.12,0.31,0.17}
\definecolor{csugold}{rgb}{0.78,0.76,0.45}
\definecolor{bucknellorange}{rgb}{0.91,0.46,0.13}

\tikzset{pics/.cd,
cube/.style args={#1/#2/#3/#4}{code={
\coordinate (O) at (0,0,0);
\coordinate (A) at (0,#2,0);
\coordinate (B) at (0,#2,#3);
\coordinate (C) at (0,0,#3);
\coordinate (D) at (#1,0,0);
\coordinate (E) at (#1,#2,0);
\coordinate (F) at (#1,#2,#3);
\coordinate (G) at (#1,0,#3);
\draw[black,fill=black!80] (O) -- (C) -- (G) -- (D) -- cycle;
\draw[black,fill=black!30] (O) -- (A) -- (E) -- (D) -- cycle;
\draw[black,fill=black!10] (O) -- (A) -- (B) -- (C) -- cycle;
\draw[black,fill=black!20,opacity=0.8] (D) -- (E) -- (F) -- (G) -- cycle;
\draw[black,fill=black!20,opacity=0.6] (C) -- (B) -- (F) -- (G) -- cycle;
\draw[black,fill=black!20,opacity=0.8] (A) -- (B) -- (F) -- (E) -- cycle;
\node at (0.5*#1,0.5*#2,0.5*#3) {#4};
}}}
\tikzset{pics/.cd,
   tcube/.style args={#1/#2/#3/#4/#5}{
      code={
         \coordinate (O) at (0,0,0);
         \coordinate (A) at (0,#2,0);
         \coordinate (B) at (0,#2,#3);
         \coordinate (C) at (0,0,#3);
         \coordinate (D) at (#1,0,0);
         \coordinate (E) at (#1,#2,0);
         \coordinate (F) at (#1,#2,#3);
         \coordinate (G) at (#1,0,#3);
         \draw[black,fill=black!5,opacity=0.6] (O) -- (A) -- (B) -- (C)-- cycle;
         \draw[black,fill=black!5,opacity=1] (O) -- (A) -- (E) -- (D)-- cycle;
         \draw[black,fill=black!5,opacity=0.8] (O) -- (D) -- (G) -- (C)-- cycle;
         \node at (0.5*#1,0.6*#2,0.5*#3) {{\small #5}};

         \draw[dotted] (E) -- (F) -- (G);
         \draw[dotted] (B) -- (F);
      }
   }
}
\tikzset{pics/.cd,
ccube/.style args={#1/#2/#3/#4}{code={
\coordinate (O) at (0,0,0);
\coordinate (A) at (0,#2,0);
\coordinate (B) at (0,#2,#3);
\coordinate (C) at (0,0,#3);
\coordinate (D) at (#1,0,0);
\coordinate (E) at (#1,#2,0);
\coordinate (F) at (#1,#2,#3);
\coordinate (G) at (#1,0,#3);
\draw[black,fill=#4!80] (O) -- (C) -- (G) -- (D) -- cycle;
\draw[black,fill=#4!30] (O) -- (A) -- (E) -- (D) -- cycle;
\draw[black,fill=#4!10] (O) -- (A) -- (B) -- (C) -- cycle;
\draw[black,fill=#4!20,opacity=0.8] (D) -- (E) -- (F) -- (G) -- cycle;
\draw[black,fill=#4!20,opacity=0.6] (C) -- (B) -- (F) -- (G) -- cycle;
\draw[black,fill=#4!20,opacity=0.8] (A) -- (B) -- (F) -- (E) -- cycle;
}}}
\tikzset{pics/.cd,
linecube/.style args={#1/#2/#3}{code={
\coordinate (O) at (0,0,0);
\coordinate (A) at (0,#2,0);
\coordinate (B) at (0,#2,#3);
\coordinate (C) at (0,0,#3);
\coordinate (D) at (#1,0,0);
\coordinate (E) at (#1,#2,0);
\coordinate (F) at (#1,#2,#3);
\coordinate (G) at (#1,0,#3);
\draw[black] (O) -- (D) -- (E) -- (F) -- (B) -- (C) -- cycle;
\draw[dotted] (D) -- (G) -- (F);
\draw[dotted] (C) -- (G);
}}}

\tikzset{pics/.cd,
lwing/.style args={#1/#2/#3/#4}{code={
\coordinate (O) at ( 0, 0, 0);
\coordinate (A) at ( 0, 0,#3);
\coordinate (B) at ( 0,#2,#3);
\coordinate (C) at ( 0,#2, 0);
\draw[black,fill=red!20] (O) -- (A) -- (B) -- (C) -- cycle;
\node at (0,0.5*#2,0.5*#3) {#4};
}}}
\tikzset{pics/.cd,
mwing/.style args={#1/#2/#3/#4}{code={
\coordinate (O) at ( 0, 0, 0);
\coordinate (A) at (#1, 0, 0);
\coordinate (B) at (#1,#2, 0);
\coordinate (C) at ( 0,#2, 0);
\draw[black,fill=green!20] (O) -- (A) -- (B) -- (C) -- cycle;
\node at (0.5*#1,0.5*#2,0) {#4};
}}}
\tikzset{pics/.cd,
rwing/.style args={#1/#2/#3/#4}{code={
\coordinate (O) at ( 0, 0, 0);
\coordinate (A) at (#1, 0, 0);
\coordinate (B) at (#1, 0,#3);
\coordinate (C) at ( 0, 0,#3);
\draw[black,fill=blue!20] (O) -- (A) -- (B) -- (C) -- cycle;
\node at (0.5*#1,0,0.5*#3) {#4};
}}}

\pgfmathsetmacro{\xx}{0.5}
\tikzset{pics/.cd,
    nilcube/.style args={#1/#2}{
        code={
            \pic at (0,-#1,0) {ccube={#1/#1/#1/#2}};
                \pic at (#1,0,0) {ccube={#1/#1/#1/#2}};
            \pic at (0,0,#1) {ccube={#1/#1/#1/#2}};
            \pic at (0,#1,0) {linecube={2*#1/-2*#1/2*#1}};
        }
    }
}

\DeclareDocumentCommand \braket { o m m } {
        \IfNoValueTF {#3} {
            \left\langle #1\, \middle|\, #2 \right\rangle
        }{
            \left\langle #2\,\middle|_{#1}\, #3 \right\rangle
        }
    }

\usepackage{algorithm}
\usepackage[noend]{algpseudocode}

\newcommand{\hm}{\scalebox{0.5}[1.0]{$-$}}

\title{Detecting sparsity patterns in tensor data}

\author{Peter A. Brooksbank}
\address{Department of Mathematics, Bucknell University, Lewisburg, PA 17837, USA}
\email{pbrooksb@bucknell.edu}
\subjclass{20D45, 20D15, 08C05, 68Q65}
\author{Martin D. Kassabov}
\address{Department of Mathematics, 310 Malott Hall, Ithaca, NY 14853}
\email{kassabov@math.cornell.edu}
\author{James B. Wilson}
\address{Department of Mathematics, Colorado State University, Fort Collins, Colorado, 80523-1874, USA}
\email{James.Wilson@ColoState.Edu}

\date{\today}

\begin{document}

\begin{abstract}
    This article introduces a class of efficiently computable 
    \textit{sparsity patterns} for tensor data. The class includes familiar patterns such as
    block-diagonal decompositions explored in statistics and signal processing 
    \citelist{
        \cite{chatroom} \cite{CL:blocksnoise} \cite{CL:blockalgebra}\cite{CS:blocks}\cite{Cardoso}}, 
    low-rank tensor decompositions \cite{DeL:Blocks}, and Tucker decompositions \cite{Tucker}. 
    It also includes a new family of sparsity patterns---not known to be detectable 
    by current methods---that can be thought of as
    continuous decompositions approximating curves and surfaces.
    We present a general algorithm to detect sparsity patterns in each class using a parameter 
    we call a \textit{chisel} that tunes the search to patterns of a prescribed shape. 
    We also show that the patterns output by the algorithm 
    are essentially unique.
\end{abstract}

\maketitle


\section{Introduction}
\label{sec:intro}
This paper concerns the detection of \emph{sparsity
patterns} in tensor data. 
A \emph{tensor} is any data with an
interpretation $\gamma: U_{\hm m}^*\times \cdots \times
U_{\hm 1}^*\times U_1\times \cdots \times U_{\ell} \to U_0$ 
as a multilinear mapping of vector spaces
$U_0,\ldots, U_{\ell}$ and dual spaces 
$U_{\hm 1}^*,\ldots, U_{\hm m}^*$, collectively called \textit{modes}, 
into a space $U_0$ called the 
\textit{base}, cf.~\cite{Lim-tensors}*{Definition 3.3}.  

A \textit{sparsity pattern} for $\gamma$ exhibited by decompositions
\begin{align*}
    U_a & = X_{a1}\oplus X_{a2}\oplus \cdots \oplus X_{a k_a}, & -m\leq a\leq \ell,
\end{align*}
is a set $\Delta \subset \prod_a \{1,\ldots,k_a\}$ satisfying
\begin{align}
    \label{eq:sparsity-pattern}
    (i_{\hm m},\ldots,  i_{\ell}) \not\in \Delta
    & \quad \implies \quad 
    \gamma(X^*_{\hm m i_{\hm m}},\ldots, X^*_{\hm 1 i_{\hm \raisebox{1pt}{\tiny $1$}}},
    X_{1 i_1},\ldots, X_{\ell i_{\ell}})\subseteq X_{0i_0}.
\end{align}
Fixing a basis $\{e_{a,i_a} \mid 1\leq i_a\leq d_a\}$ for each $U_a$, one can
represent $\gamma$ as a multiway array $\Gamma$ with entries
\begin{align}
 \Gamma_{i_1\cdots i_{\ell}}^{i_{\hm \raisebox{1pt}{{\tiny $m$}}}\cdots i_{\hm \raisebox{1pt}{{\tiny $1$}}}i_0}
    := 
        e_{0 i_0}^{\top}\gamma (
            e^{\top}_{\hm m\, i_{\hm \raisebox{1pt}{{\tiny $m$}}}},\ldots, 
                            e^{\top}_{\hm 1\, i_{\hm \raisebox{1pt}{{\tiny $1$}}}} 
                            ,
                            e_{1 i_1},\ldots, e_{\ell i_{\ell}}),      
\end{align}
where $e\mapsto e^{\top}$ is the transpose map $U\to U^*$.
Tensors on three modes can be visualized, as in Figure~\ref{fig:main-block}, 
as point clouds in 3-dimensional space. 

\begin{figure}[!h]
    \begin{center}
        \vspace{-2.2ex}
        \hfill 
        \begin{subfigure}[h]{0.19\textwidth}
            \centering
            \includegraphics[height=2cm]{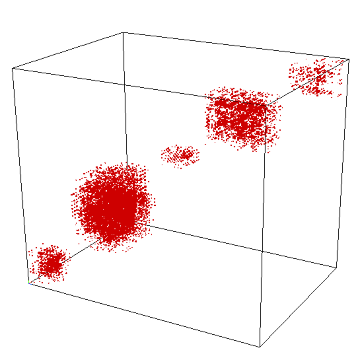}
            \caption{}
        \end{subfigure}
        \hfill 
        \begin{subfigure}[h]{0.19\textwidth}
            \centering
            \includegraphics[height=2cm]{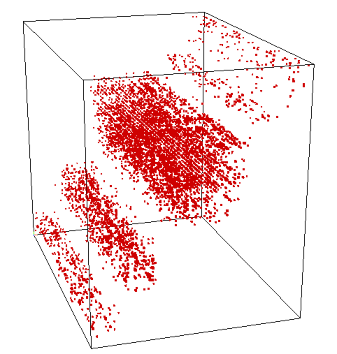}
            \caption{}
        \end{subfigure}
        \hfill
        \vspace{-2ex}
        \begin{subfigure}[h]{0.19\textwidth}
            \centering
            \includegraphics[height=2cm]{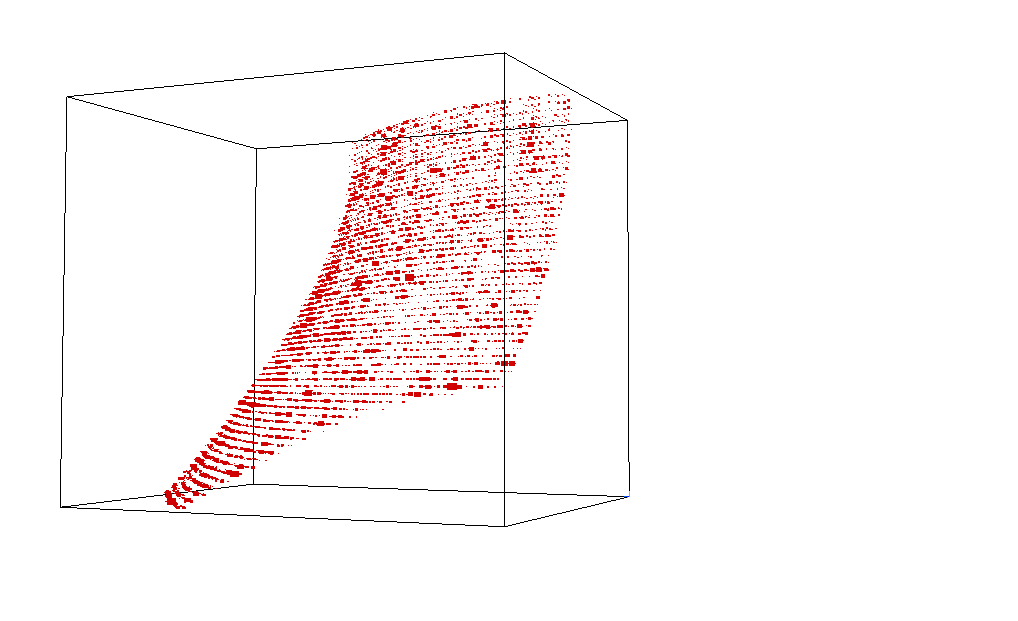}
            \caption{}
        \end{subfigure}
        \hfill
        \begin{subfigure}[h]{0.19\textwidth}
            \centering
            \includegraphics[height=2cm]{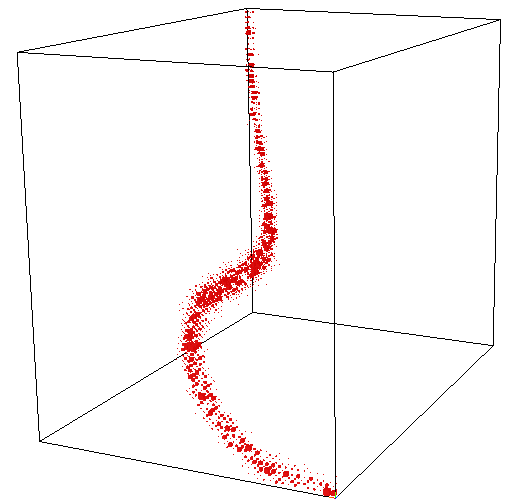}
            \caption{}
        \end{subfigure}
        \hfill
    \end{center}
    \caption{{\small Point clouds of tensor arrays recovered by our algorithms that exhibit 
    various types of sparsity patterns.}}
    \label{fig:main-block}
    \vspace{-2ex}
\end{figure}

Each of the plots reveals a sparsity pattern exhibited by some decompositions 
of the three modes. For instance, the point cloud in Figure~\ref{fig:main-block}(a) 
illustrates the sparsity pattern $\Delta=\{(i,i,i)\mid 1\leq i\leq 5\}$ 
exhibited by decompositions of each mode into five direct summands.
The surface and curve in Figure~\ref{fig:main-block}(c) 
and Figure~\ref{fig:main-block}(d), respectively, 
show sparsity patterns exhibited by more granular decompositions. 
\smallskip

It is important to emphasize that tensors are generally not
recorded in reference frames that naturally exhibit sparsity patterns, 
so the critical task is to find suitable basis changes 
for the modes and base of the tensor.
The images in Figure~\ref{fig:main-block} were generated by our 
algorithms, which found such basis changes in synthetically created and scrambled tensor data. 

\subsection*{Block decompositions}
Many problems in tensor research reduce to finding decompositions 
of the mode spaces that exhibit block structures of one form or 
another~\cite{Bader-Kolda}.   
For instance, block diagonal structures 
of the sort depicted in Figure~\ref{fig:main-block}(A)
are sought in blind source separation \citelist{\cite{Cardoso}}, 
where higher-order statistics such as moments or cumulants of independent 
events occur as distinct clusters on the diagonal. 
Figure~\ref{fig:main-block}(B) exemplifies
another common type of decomposition that focuses on 
a pair of modes and seeks decompositions along their common face. 
There are many applications of ``face block" decompositions, including 
polynomial and algebra 
factorization~\citelist{\cite{Eberly-Giesbrecht} \cite{Wilson:unique} \cite{Wilson:direct-decomp} \cite{Wilson:central}},
community detection \citelist{\cite{chatroom}\cite{chatroom2}}, simultaneous block
diagonalization of matrices \citelist{\cite{CL:blocksnoise}\cite{MM:star-alge-blocks}
\cite{CL:blockalgebra}\cite{CS:blocks}\cite{Wilson:unique}}, and low-rank
approximation problems \citelist{\cite{DeL:Blocks}}.
 
\subsection*{Continuous decompositions}
Our algorithms work for a broader class of computable sparsity patterns
that also includes the point clouds depicted 
in~Figure~\ref{fig:main-block}(C) and (D) in which the
data clusters around 
a surface, or intersection of surfaces. One can think of the corresponding 
tensor decompositions as continuous analogues of the more traditional block decompositions.

\subsection*{Features of our approach} 
Our approach is inspired by distributive products in 
the theory of (nonassociative) rings,
where Lie theory plays an essential role. In a sense, our results are the
outcomes of a study of ``regular semisimple elements of the Lie algebra
of derivations of a tensor''~\cite[Chapter IV]{Jac}. 
The result is a single algorithm that can be tuned by 
parameters to control the shape of the sparsity pattern 
sought, and also to account for additional symmetries that may be present.
Indeed, our approach unifies the outcomes of various existing 
tensor methods, such as Tucker decompositions~\cite{Tucker}, 
and some cases of Higher-Order Singular Value 
Decompositions~\citelist{\cite{DeL:Blocks}\cite{DeL:HOSVD}}. 
Details of how our methods can revover these decompositions are provided in 
Section~\ref{sec:tucker}, but we stress that in no way do we 
propose to compete with these specialized and highly 
optimized tensor techniques.

\subsection*{Implementation}
We implemented our algorithms in the Julia programming language~\cite{Julia}. 
Source code and documentation are available on GitHub~\cite{OpenDleto}, 
as well as interactive notebooks that illustrate 
the main features of our methods. Versions of 
experiments with synthetic data 
discussed throughout the paper can be reproduced 
using the notebooks. 
Our algorithms can uncover approximations to sparsity patterns,  
function reasonably well under a range of noisy conditions, 
and tolerate a number of outliers to a sparsity pattern. 
The notebooks themselves include options to incorporate a basic 
noise model into the experiment, and  Nick Vannieuwenhoven has 
recently studied the stability of variations of our algorithms 
under noise~\cite{Nick:Chisel}.

\subsection*{Outline of paper}
In Section~\ref{sec:multiarrays} we discuss the 
popular multiarray data structures 
for tensors and compare two natural tensor 
interpretations that use them. Section~\ref{sec:sparsity-patterns} 
shows how to visualize sparsity patterns in 
multiarrays using tensor contractions. 

In Section~\ref{sec:chisels},
we introduce chisels and define the class of sparsity patterns 
our methods are designed to detect. Section~\ref{sec:derivations} 
introduces \textit{derivations} as a tool to connect 
chisels and sparsity patterns to the underlying mode decompositions, 
and proves results (Theorems~\ref{thm:sparsity-to-der} and~\ref{thm:basic-chisel})
that ensure the correctness of our methods. 

In Section~\ref{sec:chiseling}, we 
present our `chiseling' algorithms: given 
a tensor and a chisel, we determine, constructively, whether there is a decomposition 
of the modes of the tensor that exhibits a sparsity pattern corresponding 
to the chisel. Section~\ref{sec:tools} takes a closer look at certain
chisels we have found useful in exploring tensors, 
and explains why they reveal the patterns for which they are designed. 
It also elucidates the connection between our methods and established 
tensor decomposition techniques.

The remainder of the paper contains an abridged theory of chisels and 
derivations needed to explain additional features of our algorithms.
Section~\ref{sec:theory} situates our approach within a general spectral theory 
of tensors~\cite{FMW}. We show, among other things, how our 
basic algorithm can be adapted 
to account for additional symmetries that may be present.
Section~\ref{sec:uniqueness} discusses the important issues of 
uniqueness and reproducibility. We show that 
algebraic properties of derivations of a `universal' chisel 
ensure that outputs of our algorithm over the complex 
field are unique. 


\section{Multiarrays}
\label{sec:multiarrays}
We saw in the introduction that array representations of tensors 
are very useful for visualization, and we shall continue to use them 
throughout. Thus, when we say ``$\Gamma$ is a tensor," we will 
usually be thinking that $\Gamma$ is an abstract array (in the sense 
of the Julia language, for instance), and we shall work with $\Gamma$ 
as a traditional multiarray. It is important, when 
using arrays, not to lose sight of the underlying 
multilinear interpretation. To emphasize this point,
we begin with two common tensor interpretations that use 
multiarrays. 

Given spaces $U_a=\mathbb{K}^{d_a}$ for $a\in [\ell]:=\{1,\ldots,\ell\}$, 
where each $u_a\in U_a$ is a function $[d_a]\to \mathbb{K}$
sending $i\mapsto u_{ai}$, the interpretation 
\begin{align}
    \label{def:tensor-product-interpretation}
    u_1\otimes \cdots \otimes u_{\ell} 
    & \quad := \quad (i_1,\ldots,i_{\ell})\mapsto u_{1i_1}\cdots u_{\ell i_{\ell}}. 
\end{align}
outputs $\ell$-arrays in $U_0=\mathbb{K}^{d_1\times \cdots \times d_{\ell}}$. 
Given $\Gamma\in \mathbb{K}^{d_1\times \cdots \times d_{\ell}}$, 
the interpretation
\begin{align}
    \label{def:multiarray-interpretation}
    \langle \Gamma \mid u_1,\ldots, u_{\ell} \rangle
    & \quad := \quad \sum_{i_1\in [d_1]}\cdots \sum_{i_{\ell}\in [d_{\ell}]}
        \Gamma_{i_1\cdots i_{\ell}} u_{1i_1}\cdots u_{\ell i_{\ell}}
\end{align}
outputs scalars in $U_0=\mathbb{K}$.
\smallskip

In the case $\ell=2$, $u_1\otimes u_2$ in
equation~\eqref{def:tensor-product-interpretation}
is a $(d_1\times d_2)$-matrix, while 
it is common to interpret
equation~\eqref{def:multiarray-interpretation} as 
$\langle u_1 \mid \Gamma \mid u_2 \rangle =\gamma(u_1,u_2)=\langle \Gamma\mid u_1^{\top},u_2\rangle$ 
for a $(d_1\times d_2)$-matrix $\Gamma$.
Figure~\ref{fig:viz-tensor} illustrates both forms of tensor interpretation 
for the case $\ell=3$, where
vectors (1-arrays) are drawn as colored ribbons, matrices (2-arrays) are rectangles, and 3-arrays are boxes.

\begin{figure}[!thbp]
    \begin{center}
        \includegraphics[width=0.9\textwidth]{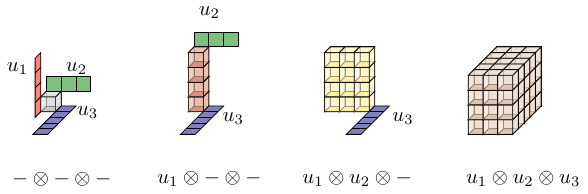}\\
        \includegraphics[width=0.9\textwidth]{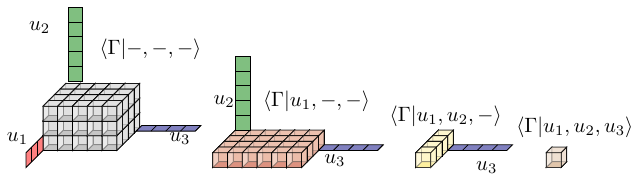}
    \end{center}
    \caption{Tensor interpretations visualized as sequences of partial evaluations. On the top 
    is the interpretation in equation~\eqref{def:tensor-product-interpretation} and on the bottom is  
    the interpretation in equation~\eqref{def:multiarray-interpretation}.}
    \label{fig:viz-tensor}
\end{figure}

\subsection{Partial evaluations of tensors}
In the same way that a matrix product can be computed as a sequence of binary products,
so multilinear maps can be computed in a sequence of partial evaluations.
Let $A=\{a_1<\cdots < a_{\ell}\}\subset [\ell]$ be a set
of modes, and let $u_a\in U_a$ for $a\in A$. 
A \textit{partial evaluation} at
$u_A=(u_a : a\in A)$, sometimes called a \textit{contraction}, 
is defined as follows:
\begin{align}
    \label{def:contract}
    \langle \Gamma \mid u_A\rangle := \langle \Gamma \mid -,u_{a_1},-\cdots-,u_{a_k},-\rangle.
\end{align}
Observe that partial evaluation produces a new tensor interpretation with modes $\{U_a : a\not\in A\}$. 
Figure~\ref{fig:viz-tensor} illustrates sequences of partial evaluations
for the multiarray interpretations in 
equations~\eqref{def:tensor-product-interpretation} and~\eqref{def:multiarray-interpretation}. 
Observe that partial evaluations reduce the valence of the tensor but 
not necessarily the size of the data structure.

\subsection{Notational convention}
\label{sec:convention}
For expedience and clarity we shall
henceforth work just with tensor interpretations like the one in 
equation~\eqref{def:multiarray-interpretation} with no 
dual spaces and a 1-dimensional base. When all vector 
spaces are finite-dimensional, one can always pass to an 
equivalent interpretation of this form. 

Specifically, given a multilinear function 
$\gamma: U_{\hm m}^*\times \cdots \times
U_{\hm 1}^*\times U_1\times \cdots \times U_{\ell} \to U_0$, 
we use $\Gamma\in\K^{d_{m+\ell}\times\cdots\times d_1\times d_0}$ 
with interpretation
\begin{align*}
\langle \Gamma\mid u_{\hm m}^{\top},\ldots,u_{\hm 1}^{\top},u_0^{\top},u_1,\ldots,u_{\ell}\rangle, 
&& \mbox{where} \quad
u_0=\gamma(u_{\hm m},\ldots,u_{\hm 1},u_1,\ldots,u_{\ell}).
\end{align*}

\subsection{Changing bases}
The original mode bases relative to which a tensor is recorded 
are considered arbitrary. In particular, even if a multilinear map  $\gamma$ 
exhibits a sparsity pattern for some decompositions of its modes, there is no 
reason to suppose that our chosen multiarray representation $\Gamma$ of $\gamma$ 
is recorded relative to bases that respect these decompositions. 
Indeed the main task, for a given $\Gamma$, is to find bases of the modes that  
determine the decompositions of some sparsity pattern for $\Gamma$, and then to rewrite 
$\Gamma$ relative to these bases.
More generally we can compose 
$\langle  \Gamma \mid \ldots\rangle$ with linear maps $X_a:V_a\to U_a$, 
defining a new tensor $\Gamma^X$ on axes $V_1,\ldots,V_{\ell}$ by
\begin{align*}
   \langle \Gamma^X\mid v_1,\ldots,v_{\ell}\rangle := \langle \Gamma\mid X_1 v_1,\ldots, X_{\ell} v_{\ell}\rangle.
\end{align*}

\section{Sparsity patterns}
\label{sec:sparsity-patterns}
Adopting the notational convention 
in Section~\ref{sec:convention} we can simplify  
condition~\eqref{eq:sparsity-pattern} and define 
a \textit{sparsity pattern} for a tensor $\Gamma$ exhibited by decompositions
\begin{align}
    \label{eq:mode-decomp}
    U_a & = X_{a1}\oplus X_{a2}\oplus \cdots \oplus X_{a k_a}, & a\in [\ell],
\end{align}
to be a set $\Delta \subset \prod_a [k_a]$ satisfying
\begin{align}
    \label{eq:null-pattern}
    (i_1,\ldots,  i_{\ell}) \not\in \Delta
    & \quad \implies \quad 
    \langle \Gamma \mid 
    X_{1 i_1},\ldots, X_{\ell i_{\ell}} \rangle =0,
\end{align}
where $\langle \Gamma \mid X_{1 i_1},\ldots, X_{\ell i_{\ell}}\rangle$ 
is the span of all partial evaluations at all 
$u\in\prod_{a\in[\ell]} X_{a i_a}$.  

Writing $\Gamma$ as a multiarray with respect to bases of the modes 
that align with the decompositions in~\eqref{eq:mode-decomp}, 
the entries of $\Gamma$ 
form blocks corresponding to tuples $(i_1,\ldots,i_{\ell})$ indexing the summands.
The tuples in the sparsity pattern $\Delta$ correspond to blocks that 
are permitted to have nonzero entries.
Every subset of $[k_1]\times\cdots \times [k_{\ell}]$ 
occurs as the sparsity pattern of some tensor, and a single tensor may have 
several sparsity patterns corresponding to different decompositions of its modes.

Sparsity patterns $\Delta,\Upsilon\subseteq [k_1]\times\cdots \times [k_{\ell}]$ 
are considered \textit{equivalent} if, for each $a\in[\ell]$, there is a  
permutation $\sigma_a$ of $[k_a]$ such that
\begin{align*} 
    \Upsilon &=\{(\sigma_1(i_1),\ldots, \sigma_{\ell}(i_{\ell})) : (i_1,\ldots, i_{\ell})\in \Delta\}.
\end{align*}

We can visualize sparsity patterns in multiarray tensors 
by computing, for each $a\in[\ell]$, contractions on some ordered basis 
of $X_{ai_a}$, where $(i_1,\ldots,i_{\ell})\in \Delta$. As the spaces 
$X_{ai_a}$ decompose the mode $U_a$, the contractions $U_a$
can thus be visualized as a square matrix partitioned into $k_a$ columns.
This introduces an arbitrary ordering, and our 
visualizations of the corresponding sparsity patterns depict this order. 

\begin{figure}[!thbp]
    \begin{center}
        \vspace{-2.2ex}
        \includegraphics[width=\textwidth]{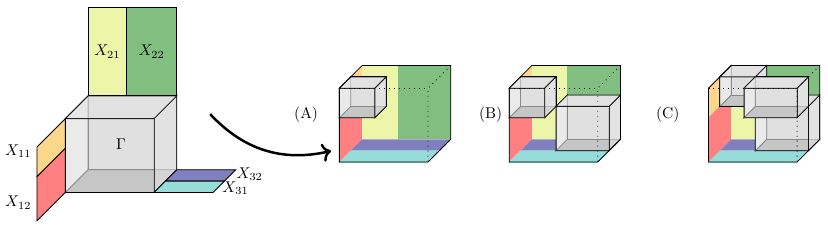}
        \vspace{-2ex}
    \end{center}
    \caption{{\small Sparsity patterns exhibited by decompositions of arrays.  
    In (A), the pattern is $\Delta=\{(1,1,1)\}$;
    in (B), $\Delta=\{(1,1,1),(2,2,2)\}$; and in (C), $\Delta=\{(1,1,2),(1,2,1), (2,2,2)\}$.
    }}
    \label{fig:structure-1}
\end{figure}

The granularity of the sparsity patterns we seek is calibrated by the dimensions 
of the spaces $X_{ai_a}$ in the decompositions. This allows sparsity patterns to 
describe a spectrum of tensor phenomena.
Figures~\ref{fig:structure-1} and~\ref{fig:curve-1} illustrate a variety of
sparsity patterns for tensors of valence 3 using shaded regions to
indicate nonzero regions. 

\begin{figure}[!htbp]
    \begin{center}
        \vspace{-2ex}
        \includegraphics[width=0.8\textwidth]{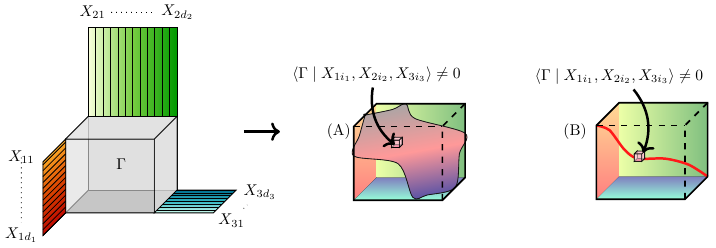}
        \vspace{-2ex}
\end{center}
\caption{Curve and surface sparsity patterns.}
\label{fig:curve-1}
\end{figure}
\medskip 

The objective, given a tensor and a sparsity pattern, is 
to find (bases for) decompositions of the modes that exhibit 
the sparsity pattern. \textit{In particular, we emphasize that full linear 
transformations are allowed on each mode; we are not constrained 
to only working with permutations of the entries 
in a multiarray.}

\section{Chisels}
\label{sec:chisels}
The definition of sparsity pattern given in the previous section is too 
broad to be practicable. It includes, for example, all higher-dimensional 
analogues of Latin square configurations. 
We focus our attention on 
families of sparsity patterns that are located on or near the 
intersection of hyperplanes. 
For this, we introduce a device we call a 
\textit{chisel}, which is an $m\times \ell$ matrix 
\begin{align}
    \label{eq:chisel_matrix}
\ch & = \left[\begin{array}{ccc}
    C_{11} & \cdots & C_{1\ell} \\
    \vdots & & \vdots \\
    C_{m1} & \cdots & C_{m\ell} \\
\end{array}\right]
\end{align} 
with entries in $\mathbb{K}$. 
The integer $\ell$ corresponds to the valence of the tensors upon which the 
chisel is deployed, and $m$ corresponds to the number of planes involved in 
the sparsity pattern. Each pair $(\ch,\delta)$, where $\ch$ is a chisel and 
\begin{align}
    \label{eq:define_delta}
\delta &=(\delta_1,\ldots,\delta_{\ell}) \in \mathbb{K}^{k_1}\times \cdots \times \mathbb{K}^{k_{\ell}}
\end{align} 
for positive integers $k_1,\ldots,k_{\ell}$, defines the sparsity pattern
\begin{align}
    \label{eq:cluster-chisel}
    \Delta(\ch,\delta) &= \left\{
        (i_1,\ldots, i_{\ell})\in \prod_{a=1}^{\ell}[k_a] ~:~ \forall j\in[m],~~
        \sum_{a=1}^{\ell}\ch_{ja}\delta_a(i_a) = 0 \right\}.
\end{align}

\begin{ex}
Consider the pair
    \begin{align}
    \label{eq:sparsity-pattern-0}
        \ch & =\begin{bmatrix}
            1 & 1 & -1  
        \end{bmatrix},
        &
        \delta &=
        \left(\begin{bmatrix} 0.5 \\ 1.0\end{bmatrix}, 
        \begin{bmatrix} 0.5\\ 1.0 \\ 2.0\end{bmatrix},
        \begin{bmatrix}
            1.0 \\ 1.5\\ 2.0\\ 2.5 
        \end{bmatrix}
        \right)\in\K^2\times\K^3\times\K^4.
    \end{align}
    Since $1\cdot \delta_1(1)+1\cdot \delta_2(2)+(-1)\cdot \delta_3(2)=0.5+1.0-1.5=0.0$, 
    we see that $\Delta(\ch,\delta)$ contains the tuple $(1,2,2)$. In fact,
    \begin{align*}
    \Delta(\ch,\delta) &=\{ (1,1,1), (1,2,2), (1,3,4), (2,1,2), (2,2,3)  \}.
    \end{align*} 
    Figure~\ref{fig:sparsity-pattern-0} plots the possible nonzero values of a 
    tensor with sparsity pattern $\Delta(\ch,\delta)$. 
    On the left is an $\ell$-uniform $\ell$-partite
    hypergraph capturing the combinations of $\delta_a(i_a)$ that sum to zero, 
    and on the right is a point cloud of a tensor showing the possible regions where the 
    points may congregate.
\end{ex}

\pgfmathsetmacro{\yy}{0.75}
\pgfmathsetmacro{\zz}{0.3}
\begin{figure}[!htbp]
    \centering
    \begin{tikzpicture}
        \node (A) at (0,0) {\begin{tikzpicture}
            
            \pic at (0,0,3*\zz) {ccube={\yy/\xx/\zz/red}};
            \pic at (\yy,0,2*\zz) {ccube={\yy/\xx/\zz/green}};
            \pic at (2*\yy,0,1*\zz) {ccube={\yy/\xx/\zz/yellow}};
            \pic at (0,\xx,2*\zz) {ccube={\yy/\xx/\zz/gray}};
            \pic at (1*\yy,\xx,0*\zz) {ccube={\yy/\xx/\zz/blue}};
            
            \pic at (0,2*\xx,0) {linecube={3*\yy/-2*\xx/4*\zz}};

            \node[scale=0.75] at (-0.75*\xx,0.25*\yy,3*\zz) {$1$};
            \node[scale=0.75] at (-0.75*\xx,0.75*\yy,3*\zz) {$2$};
            \node[scale=0.75] at (-1.5*\xx,0.5*\yy,3*\zz) {$i$};

            \node[scale=0.75] at (0*\xx,-1*\yy,0*\zz) {$1$};
            \node[scale=0.75] at (1.5*\xx,-1*\yy,0*\zz) {$2$};
            \node[scale=0.75] at (3*\xx,-1*\yy,0*\zz) {$3$};
            \node[scale=0.75] at (1.5*\xx,-1.5*\yy,0*\zz) {$j$};

            \node[scale=0.75] at (3.5*\xx,-1*\yy,-1*\zz) {$1$};
            \node[scale=0.75] at (3.5*\xx,-1*\yy,-2.5*\zz) {$2$};
            \node[scale=0.75] at (3.5*\xx,-1*\yy,-4*\zz) {$3$};
            \node[scale=0.75] at (3.5*\xx,-1*\yy,-5.5*\zz) {$4$};
            \node[scale=0.75] at (4.5*\xx,-1*\yy,-3.25*\zz) {$k$};
            
        \end{tikzpicture}};
        \node[scale=0.7, anchor=east] at (A.west)  {
            \begin{tikzpicture}
                \coordinate (d11) at (0, 0);\coordinate (d21) at (2, 0);\coordinate (d31) at (4, 0);
                \coordinate (d12) at (0,-1.5);\coordinate (d22) at (2,-1.5);\coordinate (d32) at (4,-1.5);
                                            \coordinate (d23) at (2,-3);\coordinate (d33) at (4,-3);
                                                                        \coordinate (d34) at (4,-4.5);

                \begin{scope}[opacity=.5]
                \draw[red!50,line width=5mm, rounded corners] (d11) edge[-,"$+$"] (d21) edge[-,"$-$"] (d31);
                \draw[green!50, line width=5mm, rounded corners] (d11) -- (d22) -- (d32);
                \draw[yellow, line width=5mm, rounded corners] (d11) -- (d23) -- (d34);

                \draw[gray!50,line width=5mm, rounded corners] (d12) -- (d21) -- (d32);
                \draw[blue!50, line width=5mm, rounded corners] (d12) -- (d22) -- (d33);
                \end{scope}                                                      
                \node[scale=0.85,anchor=center] at (0,1) {$i\in [2]$};  
                \node[scale=0.85,anchor=center] at (2, 1) {$j\in [3]$}; 
                \node[scale=0.85,anchor=center] at (4, 1) {$k\in [4]$};
                \draw (-1,0.5) -- ++(6,0);
                \node[anchor=west] at (d11) {$0.5$};  \node[anchor=center] at (d21) {$0.5$}; \node[anchor=east] at (d31) {$1.0$};
                \node[anchor=west] at (d12) {$1.0$};  \node[anchor=center] at (d22) {$1.0$}; \node[anchor=east] at (d32) {$1.5$};
                                       \node[anchor=center] at (d23) {$2.0$}; \node[anchor=east] at (d33) {$2.0$};
                                                             \node[anchor=east] at (d34) {$2.5$};
            \end{tikzpicture}
            };
        \end{tikzpicture}
        \caption{Sparsity pattern $\Delta(\ch,\delta)$ for the $(\ch,\delta)$
        in equation \eqref{eq:sparsity-pattern-0}.}
        \label{fig:sparsity-pattern-0}
\end{figure}

Our goal, for a given tensor and chisel $\ch$, is to determine mode decompositions that 
exhibit the patterns $\Delta(\ch,\delta)$ as sparsity patterns of the tensor.
The shapes of sparsity patterns can vary even with the same chisel.  For instance, 
each of the patterns in Figure~\ref{fig:cluster-patterns} arise from the chisel 
$\ch=[1,\;1,\;-1]$ for different choices of $\delta$.

\pgfmathsetmacro{\yy}{0.75}
\pgfmathsetmacro{\zz}{0.3}
\begin{figure}[!htbp]
    \centering
\begin{subfigure}[t]{0.45\textwidth}
    \centering
    \begin{tikzpicture}
        \node (A) at (0,0) {\begin{tikzpicture}
            
            \pic at (0,0,3*\zz) {ccube={\yy/\xx/\zz/gray}};
            \pic at (0,\xx,2*\zz) {ccube={\yy/\xx/\zz/gray}};
            \pic at (2*\yy,0,1*\zz) {ccube={\yy/\xx/\zz/gray}};
            \pic at (2*\yy,\xx,0*\zz) {ccube={\yy/\xx/\zz/gray}};
            \pic at (1*\yy,0,2*\zz) {ccube={\yy/\xx/\zz/gray}};
            \pic at (1*\yy,\xx,1*\zz) {ccube={\yy/\xx/\zz/gray}};
            
            \pic at (0,2*\xx,0) {linecube={3*\yy/-2*\xx/4*\zz}};
        \end{tikzpicture}};
        \node[scale=0.7, anchor=east] at (A.west)  {
            $\displaystyle
            \delta =
            \left(\begin{bmatrix} 0.5 \\ 1.0\end{bmatrix}, 
            \begin{bmatrix} 0.5\\ 1.0 \\ 1.5\end{bmatrix},
            \begin{bmatrix}
                1.0 \\ 1.5\\ 2.0\\ 2.5
            \end{bmatrix}
            \right)
            $};
        \end{tikzpicture}
        \caption{}
    \end{subfigure}
    \hspace{5mm}
    \begin{subfigure}[t]{0.45\textwidth}
        \centering 
        \begin{tikzpicture}
        \node (B) at (0,0) {\begin{tikzpicture}
            
            \pic at (0,0,3*\zz) {ccube={\yy/\xx/\zz/gray}};
            \pic at (0,\xx,2*\zz) {ccube={\yy/\xx/\zz/gray}};
            \pic at (1*\yy,\xx,1*\zz) {ccube={\yy/\xx/\zz/gray}};
            \pic at (2*\yy,0,0*\zz) {ccube={\yy/\xx/\zz/gray}};
            
            \pic at (0,2*\xx,0) {linecube={3*\yy/-2*\xx/4*\zz}};
        \end{tikzpicture}};
        \node[scale=0.7, anchor=east] at (B.west){
            $\displaystyle
            \delta =
            \left(\begin{bmatrix} 0.5 \\ 1.0\end{bmatrix}, 
            \begin{bmatrix} 0.5\\ 2.5 \\ 4.5\end{bmatrix},
            \begin{bmatrix}
                1.0 \\ 1.5\\ 3.5\\ 5.0
            \end{bmatrix}
            \right)
            $};
\end{tikzpicture}
\caption{}
\end{subfigure}

\begin{subfigure}[t]{0.45\textwidth}
    \centering
    \begin{tikzpicture}
        \node (A) at (0,0) {\begin{tikzpicture}
            \pic at (0,0,2*\zz) {ccube={\yy/\xx/\zz/gray}};
            \pic at (\yy,0,0*\zz) {ccube={\yy/\xx/\zz/gray}};
            \pic at (0,\xx,1*\zz) {ccube={\yy/\xx/\zz/gray}};
            \pic at (0,2*\xx,0) {linecube={3*\yy/-2*\xx/4*\zz}};
        \end{tikzpicture}};
        \node[scale=0.7, anchor=east] at (A.west)  {
            $\displaystyle
            \delta =
            \left(\begin{bmatrix} 0.5 \\ 1.0\end{bmatrix}, 
            \begin{bmatrix} 0.5\\ 1.5 \\ 5.5\end{bmatrix},
            \begin{bmatrix}
                0.0 \\ 1.0 \\ 1.5 \\ 2.0
            \end{bmatrix}
            \right)
            $};

        \end{tikzpicture}
        \caption{}
    \end{subfigure}
    \hspace{5mm}
    \begin{subfigure}[t]{0.45\textwidth}
        \centering 
        \begin{tikzpicture}
        \node (B) at (0,0) {\begin{tikzpicture}
            
            \pic at (0,0,3*\zz) {ccube={\yy/\xx/\zz/gray}};
            \pic at (0,\xx,2*\zz) {ccube={\yy/\xx/\zz/gray}};
            \pic at (1*\yy,0,1*\zz) {ccube={\yy/\xx/\zz/gray}};
            \pic at (1*\yy,\xx,0*\zz) {ccube={\yy/\xx/\zz/gray}};
            
            \pic at (0,2*\xx,0) {linecube={3*\yy/-2*\xx/4*\zz}};
        \end{tikzpicture}};
        \node[scale=0.7, anchor=east] at (B.west){
            $\displaystyle
            \delta =
            \left(\begin{bmatrix} 0.5 \\ 1.0\end{bmatrix}, 
            \begin{bmatrix} 0.5\\ 1.5 \\ 3.5\end{bmatrix},
            \begin{bmatrix}
                1.0 \\ 1.5\\ 2.0\\ 2.5
            \end{bmatrix}
            \right)
            $};  
\end{tikzpicture}
\caption{}
\end{subfigure}
\caption{Sparsity patterns $\Delta([1,1,-1],\delta)$ for some choices of $\delta$.}
\label{fig:cluster-patterns}
\end{figure}

Sparsity patterns corresponding to a chisel generally bear no 
resemblence to the hyperplanes defining the chisel, 
as the following example illustrates.

\begin{ex}
\label{ex:sphere-lab}
Consider the discretized surface
of a sphere with radius $r$ and center $(c_1,c_2,c_3)$, as illustrated in
Figure~\ref{fig:sphere-pattern}(A). The equation of this surface is
$(x-c_1)^2+(y-c_2)^2+(z-c_3)^2=r^2$, so if we set
$\delta_a(i_a)=(i_a-c_a)^2-r^2/3$  then $\delta_1(i)+\delta_2(j)+\delta_3(k)=0$
exactly when $(i-c_1)^2+(j-c_2)^2+(k-c_3)^2=r^2$.  Thus, the sparsity pattern
$\Delta(\ch,\delta)$ for the chisel $\ch=[1,1,1]$ describes the sphere.

\begin{figure}[!htbp]
    \centering
    \begin{subfigure}[t]{0.4\textwidth}
        \centering
    \includegraphics[width=1in]{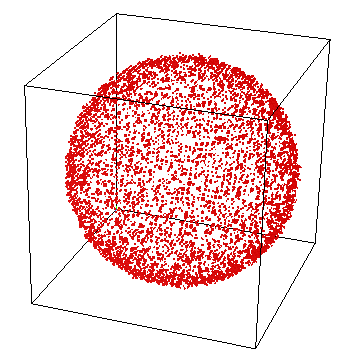}
    \caption{}
    \end{subfigure}
    \begin{subfigure}[t]{0.4\textwidth}
        \centering
    \includegraphics[width=1in]{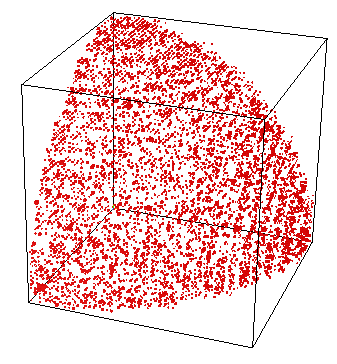}
    \caption{}
    \end{subfigure}
    \caption{Sparsity pattern $\Delta([1,1,1],\delta_a(i_a)=(i_a-c_a)^2-r^2/3)$ 
    together with its sorted level-set form.
    }
    \label{fig:sphere-pattern}
\end{figure}

The sphere pattern is clearly symmetric, with 
\begin{align*}
\delta_a(i_a) &=(\pm (i_a-c_a))^2+r^2/3=\delta_a(2c_a-i_a).
\end{align*}
By choosing a permutation $\sigma_a$ that arranges 
the sequence of scalars $\delta_a( i_{\sigma(a)})$ 
in decreasing order, we group together coordinates with the same $\delta$ values.
In that order the pattern appears as a discretized octant of the sphere 8 layers 
thick, as illustrated in Figure~\ref{fig:sphere-pattern}(B).  

One of the design choices in our algorithms is to use such 
an ordering. Our algorithms are therefore predisposed to locate sparsity patterns 
of the type depicted in Figure~\ref{fig:sphere-pattern}(B) rather than that 
in Figure~\ref{fig:sphere-pattern}(A).
Nevertheless, any tensor admitting one of these patterns necessarily admits the other.
Thus, one can view our design choice as selecting a canonical order from the 
possible orders. A notebook replicating this experiment can be found in
\cite{OpenDleto}*{SphereLab}.
\end{ex}

\section{Derivations}
\label{sec:derivations}
Chisels have so far been used only to \textit{define} 
the classes of sparsity patterns we wish to study. 
In this section we explain how they can also be used to 
\textit{detect} sparsity patterns in a given tensor. 
Historically, derivations have proved to be  useful 
when studying products. Through their interpretation 
as multilinear functions, tensors generalize distributive 
binary products.
It is therefore reasonable to explore invariants 
for tensors that behave like derivations.

\subsection{Derivations associated with chisels}
Consider a 3-tensor with identical modes $U_1=U_2=U_3$. 
This could have come, for example, from the 
(homogeneous) multiplication $*:A\times A\to A$ 
of a $\mathbb{K}$-algebra $A$. 
As in~\cite[Chapter III, Section 6]{Jac}, a \textit{derivation} 
in this context is a $\mathbb{K}$-linear function $D:A\to A$ satisfying 
\begin{align}
    \label{eq:algebra-der}
    \forall a\in A,\,\forall b\in A, && D(a*b)=D(a)*b+a*D(b).
\end{align}

Algebras can act on other spaces, leading to more general 
bilinear products of the form $*:A\times B\to C$. 
Since the operator $D$ can no longer act in the same way 
on the three spaces $A$, $B$, $C$, we must  
generalize the derivation definition accordingly. 
Denoting those actions by $D_A$, $D_B$, and $D_C$, we 
get the condition
\begin{align*}
    \forall a\in A,\,\forall b\in B, && D_C(a*b)=D_A(a)*b+a*D_B(b).
\end{align*}

In view of our notational convention (Section~\ref{sec:convention}), 
let us write
$*$ as a product $\langle \cdots \rangle:A\times B\times C^*\to \K$, where
$\langle a,b,c^{\top}\rangle:=c^{\top}(a*b)$, which allows us 
reformulate the derivation 
condition as follows:
\begin{align*}
\forall a\in A,\,\forall b\in B,\,\forall c\in C, &&
\langle D_A(a),b,c\rangle+\langle a,D_B(b),c\rangle-\langle a,b,D_C^{*}(c)\rangle=0.
\end{align*}
This suggests the following general definition.

\begin{defn}
A tuple $(D_1,\ldots,D_{\ell})$ of operators on the modes 
$U_1,\ldots,U_{\ell}$ of a tensor $\Gamma$ is a 
\emph{$\ch$-derivation} of $\Gamma$ if it satisfies the linear system 
\begin{equation}
    \label{eq:C-derivation}
    \begin{array}{cccc}    
    \ch_{11}\langle \Gamma\mid D_1 u_1,\ldots, u_{\ell}   \rangle
        & + \qquad\cdots\qquad + & \ch_{1\ell}\langle \Gamma\mid u_1,\ldots, D_{\ell}u_{\ell}\rangle 
        & = \quad 0  
        \\
        \vdots &  & \vdots  & \quad \vdots \\
    \ch_{m1}\langle \Gamma\mid D_1 u_1,\ldots, u_{\ell}   \rangle
        & + \qquad\cdots\qquad + & \ch_{m\ell}\langle \Gamma\mid u_1,\ldots, D_{\ell}u_{\ell}\rangle 
        & = \quad 0 
    \end{array}
\end{equation}
for all $(u_1,\ldots,u_{\ell})\in U_1\times\cdots\times U_{\ell}$.
Define 
\begin{align}
\label{eq:define-DerC}
\Der(\ch,\Gamma) ~&:=~ \{\,(D_1,\ldots,D_{\ell})~\mbox{satisfying system}~\eqref{eq:C-derivation}\,\}
\end{align}
to be the $\K$-space of all $\ch$-derivations of $\Gamma$. 
\end{defn}

\begin{rem}
    \label{rmk:engaged}
    Suppose, for some $a\in[\ell]$, that column $a$ of 
    $\ch$ consists entirely of zeros. Then the derivation condition 
    imposes no constraint on the operator $D_a$. In that case we say 
    that $\ch$ is \textit{not engaged} on mode $a$ and impose the 
    condition $D_a=0$.  
\end{rem}

To help digest this key definition, we describe the 
space of derivations of a tensor coming from a finite-dimensional algebra.

\begin{ex}
    \label{ex:derivation_example}
    Consider the polynomial algebra $\K[x]/(x^d-1)$. It has a 
    multiplication table relative to the basis $\{1,x,x^2,\ldots,x^{d-1}\}$
    with relations $x^{i}*x^{j}=x^{i+j}$ if $i+j<d$ and
    $x^{i}*x^{j}=x^{i+j-d}$ if $i+j\geq d$, and is therefore a 3-tensor
    represented as an array
    $\Gamma\in \K^{d\times d\times d}$, where for $1\leq i\leq d$,
    \begin{align*}
    \overset{x^i}{\Gamma[:,:,i]} & = 
    \begin{bmatrix}
    A_i & 0 \\ 0 & A_{d-i}
    \end{bmatrix}, 
    & 
    A_i = \begin{bmatrix} 0 & \cdots & 0 & 1 \\
                                     0 & \cdots & 1 & 0 \\
                                     \vdots & \reflectbox{$\ddots$} & & \vdots \\
                                     1 & 0 & \cdots & 0 
    \end{bmatrix} \in \K^{i\times i}.
    \end{align*}  
    Let $\mathbb{A}$ consist of all matrices of the form
        \begin{align*}
            \begin{bmatrix}
                \lambda_0  & \lambda_1 &  & \cdots & \lambda_{d-1}\\
                \lambda_{d-1} & \lambda_0 & \lambda_1 & \cdots & \lambda_{d-2} \\
                \vdots & \ddots & \ddots & \ddots & \vdots \\
                \lambda_2 &  & \lambda_{d-1} & \lambda_0 & \lambda_1 \\
                \lambda_1 & \lambda_2 & \cdots & \lambda_{d-1} & \lambda_0 
            \end{bmatrix}, && \lambda_i\in \K.
        \end{align*}
    Solving for $D_1,D_2,D_3$ in the linear system \eqref{eq:C-derivation} 
    with $\ch=[1,1,-1]$, we get
    \begin{align*}
        \Der([1,1,-1],\Gamma) & = \left\{
            \left(X, Y, X+Y\right) ~\middle|~
            X,Y\in\mathbb{A} \right\}.
    \end{align*}
These are the derivations of the algebra 
$\K[x]/(x^d-1)$ satisfying condition~\eqref{eq:algebra-der}.
\end{ex}

\subsection{Sparsity patterns produce derivations}
If we have a decomposition that exhibits 
a sparsity pattern for a tensor, then we can always find derivations:
\begin{thm}
\label{thm:sparsity-to-der}
    Let $\Gamma$ be an $\ell$-tensor, let $\ch$ be an 
    $\ell$-chisel, 
    and let $\delta\in \K^{k_1}\times \cdots \times
    \K^{k_{\ell}}$.
    Suppose we have decompositions 
    $U_a=X_{a 1}\oplus \cdots \oplus X_{a k_a}$ 
    of the modes that exhibit the sparsity pattern 
    $\Delta(\ch,\delta)$ for $\Gamma$. For $a\in[\ell]$, let 
    $D_a$ be the
    linear operator with $\delta_{a i_a}$-eigenspaces $X_{a i_a}$. Then
    $D=(D_1,\ldots, D_{\ell})$ is a $\ch$-derivation of $\Gamma$.    
\end{thm}

\begin{proof}
    For each $a\in[\ell]$ and 
    $i_a\in[k_a]$,
    let $x_{a i_a}$ be vectors in $X_{a i_a}$, 
    which by construction are $\delta_{a i_a}$-eigenvectors 
    of $D_a$. Hence, for each $1\leq j\leq m$, we have
    \begin{align*}
        \ch_{j1}\langle \Gamma\mid &D_1 x_{1i_1},x_{2i_2},\ldots,x_{\ell i_{\ell}}\rangle 
         +
        \cdots 
        +
        \ch_{j\ell}\langle \Gamma\mid x_{1i_1},\ldots, x_{\ell-1 i_{\ell-1}},D_{\ell} x_{\ell i_{\ell}}\rangle \\
        & = 
        \ch_{j1}\delta_{1 i_1}\langle \Gamma\mid x_{1i_1},\ldots,x_{\ell i_{\ell}}\rangle
        +
        \cdots 
        +
        \ch_{j\ell}\delta_{\ell i_{\ell}}\langle \Gamma\mid x_{1i_1},\ldots,x_{\ell i_{\ell}}\rangle\\
        & = \left(\sum_{a=1}^{\ell} \ch_{ja} \delta_{a i_a}\right)\langle \Gamma\mid x_{1 i_1},\ldots,x_{\ell i_{\ell}}\rangle.
    \end{align*}
    Equation~\eqref{eq:C-derivation} is satisfied for $u_a=x_{a i_a}$ if 
    $\langle \Gamma\mid x_{1 i_1},\dots,x_{\ell i_{\ell}}\rangle = 0$.
    If not, since the decomposition exhibits the sparsity pattern 
    $\Delta(\ch,\delta)$, we have 
    $\sum_{a=1}^{\ell} \ch_{ja} \delta_{a i_a}=0$, so 
    equation~\eqref{eq:C-derivation} is still satisfied
    for $u_a=x_{a i_a}$.
    As the spaces $X_{a i_a}$ decompose the modes of $t$, the union 
    of any choice of bases for them span the $U_a$. It follows 
    that $(D_1,\ldots,D_{\ell})$ satisfies 
    condition~\eqref{eq:C-derivation}, and hence 
    is a $\ch$-derivation of $\Gamma$.
\end{proof}

\subsection{Derivations reveal sparsity patterns}
We have just seen that sparsity patterns in tensors ensure 
the existence of derivations. Conversely, derivations 
can be leveraged to discover sparsity patterns in tensors.

\begin{thm}
    \label{thm:basic-chisel}
    Let $(D_1,\ldots,D_{\ell})$ be a diagonalizable
    $\ch$-derivation of a tensor $\Gamma$ for some chisel $\ch$.
    For each $a\in [\ell]$, let $X_{a1},\ldots,X_{ak_a}$ be the 
    eigenspaces of $D_a$ with eigenvalues 
    $\delta_{a1},\ldots,\delta_{ak_a}$.
    This eigenspace decomposition of the modes exhibits 
    the sparsity pattern 
    $\Delta(\,\ch\,,\,(\delta_1,\ldots,\delta_{\ell})\,)$ of $\Gamma$.
\end{thm}

\begin{proof}
    For $a\in[\ell]$, let $X_{ai_a}$ be the
    $\delta_{ai_a}$-eigenspace of the diagonalizable operator $D_a$. 
    Since $(D_1,\ldots,D_{\ell})$ is a 
    $\ch$-derivation, for each $j\in [m]$ and $x_{ai_a}\in X_{a,i_a}$,  
    \begin{align*}   
    0=~~&\ch_{j1}\langle \Gamma\mid D_1\,x_{1i_1},\ldots,x_{\ell i_{\ell}}\rangle+
    \cdots+
    \ch_{j\ell}\langle \Gamma\mid x_{1i_1},\ldots,D_{\ell}\,x_{\ell i_{\ell}}\rangle \\
    =~~& \ch_{j1}\langle \Gamma\mid \delta_{1i_1}x_{1i_1},\ldots,x_{\ell i_{\ell}}\rangle+
    \cdots+
    \ch_{j\ell}\langle \Gamma\mid x_{1i_1},\ldots,\delta_{\ell i_{\ell}}x_{\ell i_{\ell}}\rangle
     \\
    =~~&
    \left(\sum_{a=1}^{\ell}\ch_{ja}\delta_{ai_a}\right) \langle \Gamma\mid x_{1i_1},\ldots,x_{\ell i_{\ell}}\rangle. 
    \end{align*}
    It follows that 
    \begin{align*} 
    \langle \Gamma\mid X_{1i_1},\ldots,X_{\ell i_{\ell}}\rangle\neq 0 
    \quad &\implies \quad
    (i_1,\ldots,i_{\ell})\in \Delta(\ch,\delta),
    \end{align*}
    so the decompositions exhibit the sparsity pattern 
    $\Delta(\ch,\delta)$ in $\Gamma$.
\end{proof}

\subsection{Scalar derivations reveal nothing}
\label{sec:trivial}
Let $(D_1,\ldots,D_{\ell})$ be a derivation of a tensor $\Gamma$.
If the eigenspace decomposition of some $D_a$ consists of the 
single space $U_a$, then $D_a$ has a single eigenvalue $\delta_a$, 
so $D_a=\delta_a 1_{U_a}$.
If this is the case for every $a\in [\ell]$, 
we call $(D_1,\ldots, D_{\ell})$ a 
\textit{scalar derivation} of $\Gamma$.

Notice that if $\delta\in\K^{\ell}$ is a row vector 
with $\ch \delta^{\top}=0$, 
then $(\delta_11_{U_1},\ldots,\delta_{\ell}1_{U_{\ell}})$ 
is a scalar derivation of \textit{any} tensor $\Gamma$. Hence, the dimension 
of $\Der(\ch,\Gamma)$ is at least the dimension of the null space of $\ch$.
If $(D_1,\ldots, D_{\ell})$ is a scalar derivation, then
the sparsity pattern $\Delta(\ch,\delta)=\{(1,\ldots,1)\}$ 
is trivial. Thus, scalar derivations  
tell us nothing about the existence of nontrivial sparsity patterns
in a given tensor, so we disregard them.

\section{Chiseling}
\label{sec:chiseling}
Theorems~\ref{thm:sparsity-to-der} and~\ref{thm:basic-chisel} 
provide the foundation for
practical algorithms to detect sparsity patterns in tensors. 
In this section we present two such algorithms. The first,
Algorithm~\ref{algo:basic-stratify}, is an elementary 
field-agnostic version based directly on 
Theorem~\ref{thm:basic-chisel}.
The second, Algorithm~\ref{algo:detailed-stratify}, is a 
practical heuristic version designed for use on real data.
Our Julia implementation~\cite{OpenDleto} is based on  
Algorithm~\ref{algo:detailed-stratify}.

\subsection{Basic chiseling}
Algorithm~\ref{algo:basic-stratify} takes as input a tensor 
$\Gamma$ and a chisel $\ch$ 
and returns a decomposition of the modes of $\Gamma$ exhibiting 
a sparsity pattern $\Delta(\ch,\delta)$ for some non-scalar 
derivation $\delta\in\Der(\ch,\Gamma)$, if such exists.
\medskip

\begin{algorithm}
    \caption{{\color{blue} (chiseling in tensors --- a basic version)}}
    \begin{algorithmic}[1]
    \Require a tensor $\Gamma$, and a chisel $\ch$.

    \Ensure decompositions $U_a=X_{a1}\oplus\cdots\oplus X_{ak_a}$ of
    the modes of $\Gamma$ exhibiting a nontrivial sparsity pattern
    $\Delta(\ch,\delta)$ of $\Gamma$,
    for some $\delta$.

    \State Find a non-scalar solution $(\,D_1\,,\,\ldots\,,\,D_{\ell}\,)$ 
    to the linear system~\eqref{eq:C-derivation}, where each 
    $D_a$ is diagonalizable.  

    \State For each $a$, compute the eigenvalues 
    $\delta_a=(\,\delta_{a1}:i_a\in [k_a]\,)$ of $D_a$ and 
    the corresponding eigenspace decomposition
    $U_a=X_{a1}\oplus\cdots\oplus X_{ak_a}$. 
    \State \textbf{return} bases for the spaces $X_{ai_a}$ and the 
    sparsity pattern $\Delta(\ch,(\delta_a:a\in[\ell]))$.
    \end{algorithmic}
    \label{algo:basic-stratify}
    \end{algorithm} 

\noindent \textit{Discussion:}
When working over fields with exact operations, 
such as finite fields or the rationals,
there are several ways to accomplish Lines 1 and 2. 
The system~\eqref{eq:C-derivation} is linear so we can solve it using 
established techniques such as least-squares, LU decomposition, and 
pseudo-inverses. Algorithms for computing  
the eigenspace decompositions in Line 2 can also 
be found in standard linear algebra packages. 
There are issues even in the exact field setting---we mention two of them. 

First, how can we avoid unnecessary overhead 
in solving~\eqref{eq:C-derivation}? 
The default approach solves a linear system with
$d_1\cdots d_{\ell}$ equations in 
$v=d_1^2+ \cdots + d_{\ell}^2$ variables.  
A solution to this large and vastly over-determined system 
can proceed by sampling $(1+\varepsilon)v$ equations to solve 
an approximately square system $M$.  Furthermore, each 
equation uses $O(\sqrt{v})$ variables (far less if 
$\Gamma$ is a sparse array), so mat-vecs $Mv$ 
can be evaluated using $O(v^{1.5})$ field operations.  
An iterative solver can find a solution to $M$ in 
$O(v^{2.5})\subset O((\ell d)^5)$, where 
$d=\max\{d_1,\ldots,d_{\ell}\}$ operations.  
Since this system is nearly square, with high probability 
it has few solutions, so any solution to $Mv=0$ is 
likely to be a solution to the entire system.  

The complexity is therefore $O((\ell d)^5)$, which is expensive. 
In fact, absent additional engineering, this is quadratically 
worse than using HOSVD~\cite{DeL:HOSVD}.  
It would be better to leave the system as it is given and 
use tensor methods to solve it. A project to develop such a solver 
is currently underway. Note, however,
that even reading the data in a 
dense $d_1\times \cdots \times d_{\ell}$ 
array $\Gamma$ requires $O(d^{\ell})$ operations.  
\smallskip 

The second issue is how to select a non-scalar derivation $D$ 
from the solution space in a way that produces the best possible 
results? We comment further on this issue after we present
Algorithm~\ref{algo:detailed-stratify} below.

\subsection{Chiseling with SVDs}
Algorithm~\ref{algo:basic-stratify} is more or less 
field agnostic, but as stated it is clearly inadequate for use with 
practical models of the real and complex numbers. 
In these settings, one can only target approximate solutions 
to~\eqref{eq:C-derivation}, and one must deploy robust floating point 
algorithms. For the remainder of the section, we assume that $\K$ is 
either the real field $\R$ or the complex field $\C$. 
Algorithm~\ref{algo:detailed-stratify} is a
heuristic version of Algorithm~\ref{algo:basic-stratify} 
that works well in these settings.
\medskip

\begin{algorithm}
      \caption{{\color{blue} (chiseling in tensors --- a heuristic version)}}
      \begin{algorithmic}[1]
        \Require
        a tensor $\Gamma$, and a chisel $\ch$.

    \Ensure decompositions $U_a=X_{a,1}\oplus\cdots\oplus X_{a,k_a}$ of 
    the modes of $\Gamma$ exhibiting a sparsity pattern $\Delta(\ch,\delta)$ 
    for some non-scalar $\ch$-derivation $\delta$ of $\Gamma$.

      \State 
      Build 
      the matrix $N$ 
      whose nullspace parameterizes~\eqref{eq:C-derivation}.

      \State Compute the dimension $e$ of the 
      nullspace of $\ch$.

      \State Computing the SVD of $N$, locate
       the $(e+1)^{\text{st}}$ smallest singular value, $\sigma_{e+1}$.

      \State If $\sigma_{e+1}$ is close to 0, use the corresponding singular vector 
      to build matrices $(\,D_1\,,\,\ldots\,,\,D_{\ell}\,)$ that approximately 
      satisfy~\eqref{eq:C-derivation}. Otherwise report that $\Gamma$ 
      does not conform to any sparsity pattern determined by $\ch$.
    
      \State For $a\in [\ell]$, compute invertible $X_a$ such that 
      $D_a X_a=X_a \text{diag}(\delta_{a1},\ldots,\delta_{ak_a})$.
      
      \State \textbf{return} $(X_a : a\in [\ell])$
      and the sparsity pattern $\Delta(\ch,(\delta_a:a\in[\ell]))$.
      \end{algorithmic}
      \label{algo:detailed-stratify}
\end{algorithm} 

\noindent \textit{Discussion:}
The SVD of $N$ produces approximate solutions to the nullspace by
inspecting the singular values closest to $0$.  
One key advantage of SVD is that the singular values are ordered. This 
provides a canonical choice of vector from which to build a `good' derivation 
$D$ in Line 1 of Algorithm~\ref{algo:basic-stratify}, as we now explain. 

Recall that each chisel $\ch$ has its own nullspace of dimension $e$, 
and the vectors in this nullspace give rise to scalar derivations. 
As these involve scalar matrices, they arise as solutions 
to~\eqref{eq:C-derivation} with very little error---they correspond in 
the SVD to the singular values closest to 0. We therefore skip to the 
$(e+1)^{\text{st}}$ smallest singular 
value, $\sigma_{e+1}$, and regard $\sigma_{e+1}$ being close to 0 as a proxy 
for $\Gamma$ having a sparsity pattern of the type detected by the chisel $\ch$.
The corresponding singular vector produces 
$(\,D_1\,,\,\ldots\,,\,D_{\ell}\,)$ 
that approximates a non-trivial $\ch$-derivation of $\Gamma$. 
Thus, proceeding in this way removes the search steps in 
Algorithm~\ref{algo:basic-stratify}. 

\begin{rem}
Algorithms~\ref{algo:basic-stratify} and~\ref{algo:detailed-stratify}
only use diagonalizable operators. This is usually the case in practical 
settings, but one can adapt the technique---whose correctness uses 
Theorem~\ref{thm:basic-chisel}---to handle generalized eigenspaces. 
Indeed, there are situations where it is interesting to 
consider nilpotent derivations, but we only focus on the semisimple 
case in this article.
\end{rem}

\begin{rem}
Based on experiments with the Julia implementation, 
our methods tolerate a certain amount of noise in the data.
Recently, Nick Vannieuwenhoven has conducted a more robust
analysis of variations of Algorithm~\ref{algo:detailed-stratify} under 
noise~\cite{Nick:Chisel}.
\end{rem}

\subsection{Implementation}
\label{sec:Julia}
We have implemented versions of 
Algorithm~\ref{algo:detailed-stratify} 
in the Julia language~\cite{Julia}, 
and the code is available at~\cite{OpenDleto}. 
The implementation makes substantial use of the ITensor 
Software Library~\cite{ITensor}.
The GitHub repository~\cite{OpenDleto} contains 
Jupyter notebooks such as~\cite{OpenDleto}*{SphereLab}, which was 
mentioned in Example~\ref{ex:sphere-lab}. 
The notebook~\cite{OpenDleto}*{Chiseling101} 
demonstrates the basic features of the implementation
and allows users run versions of the 
experiments we describe in the next section. 
The notebook~\cite{OpenDleto}*{WhatWeEatInAmerica} 
illustrates chiseling on real data sets.
The notebooks are intended to help researchers explore 
chiseling, whether their own tensors
arise from numerical examples, synthetic 
data, or real data.
They also show how to add a basic noise model to 
experiments using optional parameters. 

\section{Chiseling 3-tensors: a user's guide}
\label{sec:tools}
A key feature of Algorithms~\ref{algo:basic-stratify} 
and~\ref{algo:detailed-stratify} is the control they
afford users to search for the type of structure
they wish to find in their tensor data---or, more likely, that 
they believe exists within their data. The control comes 
from the ability to choose the chisel matrix $\ch$, 
and in this section we describe several choices.
\smallskip

Sections~\ref{sec:universal}---\ref{sec:centroid} 
introduce three chisels that we have found particularly useful
in our experiments. We describe the sparsity patterns that each 
chisel $\ch$ reveals, and give demonstrations using our Julia 
implementation on synthetically created tensors. 
The demonstrations involve tensors whose three mode dimensions 
are between 40 and 90, and each one took a couple of minutes 
to run on a standard laptop.
None of the experiments incorporated noise 
models, but we have conducted tests that do; 
see also~\cite{Nick:Chisel} for a robust 
discussion of noise tolerance.

In Section~\ref{sec:reasons} we explain why these 
particular chisels reveal 
the corresponding sparsity patterns in tensors, and in 
Section~\ref{sec:tucker} we show how the familiar 
Tucker decomposition can be carried out by chiseling, 
as well as some cases of HOSVD.
\smallskip

For visualization reasons we confine our attention to 3-tensors 
throughout this section, but the properties we discuss 
generalize to higher valence tensors.

\subsection{The universal chisel}
\label{sec:universal}
We first introduce the chisel that serves as a default chisel---for 
reasons explained by Proposition~\ref{prop:universal-chisel},
we call it the \textit{universal chisel}:
\begin{align}
    \label{eq:universal-chisel}
    \ch_{{\rm uni}} &= \left[
        \begin{array}{ccc}
            1 & 1 & 1 
        \end{array}
    \right].
\end{align}

The solution space $\Der(\ch_{{\rm uni}},\Gamma)$ of~\eqref{eq:C-derivation} 
using chisel $\ch_{{\rm uni}}$ forms a Lie algebra, where 
the Lie bracket is the commutator map~\cite[Theorem A]{BMW:Rihanna}. 
We use this fact in the next section to establish uniqueness.
The scalar $\ch_{{\rm uni}}$-derivations are given by
\begin{align} 
\label{eq:scal-der}
\mathrm{scl}(\ch_{\mathrm{uni}}) &= 
    \{~(\lambda_1 I_{d_1},\lambda_2 I_{d_2},\lambda_{d_3} I_{d_{3}})~:~\lambda_1+\lambda_2+\lambda_{3}=0\,\},
\end{align}
a 2-dimensional subspace of $\Der(\ch_{\mathrm{uni}},\Gamma)$. Hence, when 
using $\ch_{\mathrm{uni}}$ in Algorithm~\ref{algo:detailed-stratify}, 
we use the third smallest singular value
of the matrix $N$ in Line 3.

The universal chisel $\ch_{\mathrm{uni}}$ 
was used to reveal the sparsity patterns depicted
in Figure~\ref{fig:main-block}(C) and (D).
Experiments that recover images similar to these 
are described in Figures~\ref{fig:der-chisel-surface} 
and~\ref{fig:der-chisel-sample}, respectively.
\vspace*{4mm}

\begin{figure}[!htbp]
    \centering
    \begin{subfigure}[t]{0.3\textwidth}
       \centering
        \vspace{-3ex}
            \includegraphics{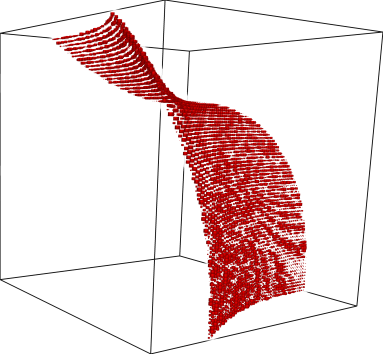}    
    \end{subfigure}
    \begin{subfigure}[t]{0.32\textwidth}
        \centering
        \vspace{-3ex}
        \includegraphics{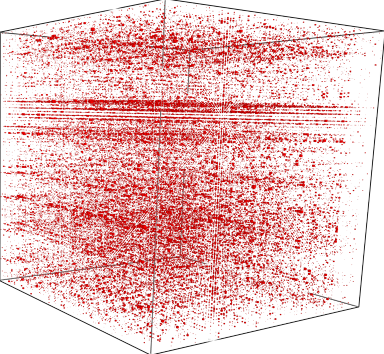}
    \end{subfigure}   
    \begin{subfigure}[t]{0.3\textwidth}
       \centering
        \vspace{-3ex}
        \includegraphics{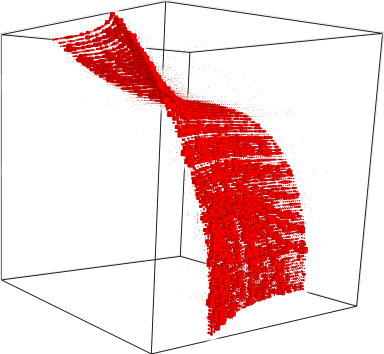}   
    \end{subfigure}
    \caption{{\small The $81\times 81\times 81$ 
    tensor on the left has support constrained near a surface. 
    The tensor was then scrambled with random basis changes 
    to produce the middle tensor. It was then
    passed to Algorithm~\ref{algo:detailed-stratify} with chisel 
    $\ch_{\mathrm{uni}}$, and it returned the tensor on the right.}}
    \label{fig:der-chisel-surface}
    \end{figure}

    \begin{figure}[!htbp]
    \centering
    \begin{subfigure}[t]{0.4\textwidth}
        \centering
        \includegraphics{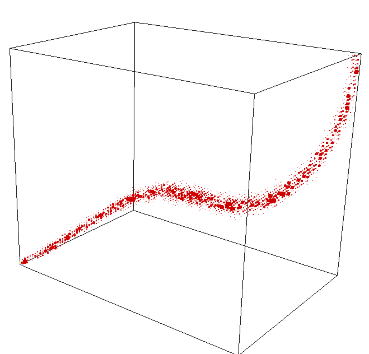}
    \end{subfigure}   
    \begin{subfigure}[t]{0.4\textwidth}
       \centering
        \includegraphics{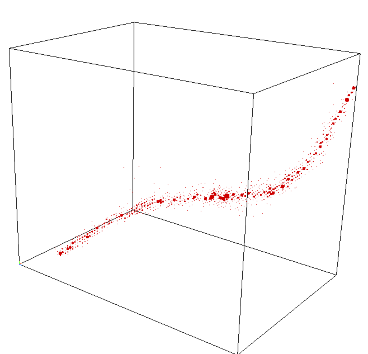}
        \end{subfigure}
    \caption{{\small The $(40\times 60\times 70)$-array 
    on the left was constructed to have support near a curve. 
    The array was again scrambled (image omitted) as in 
    Figure~\ref{fig:der-chisel-surface}.  
    The scrambled array was then passed to our algorithm, 
    which reconstructed the array on the right.}}
    \label{fig:der-chisel-sample}
    \end{figure} 

Recall from Example~\ref{ex:sphere-lab} that~\cite{OpenDleto}*{SphereLab} 
is a notebook that demonstrates how to use our algorithm to recover sparsity patterns 
for data congregated near the surface of a sphere. It uses 
the universal chisel to do this and can be adapted to experiment with 
other ``continuous'' sparsity patterns.

\subsection{Adjoint chisels}
\label{sec:adjoint}
The tensor depicted in Figure~\ref{fig:main-block}{\small {\sc (B)}}
is a decomposition with respect to the front face. Specifically,
it exhibits a sparsity pattern for decompositions of just 
axes 1 and 2 into three nontrivial subspaces. 
Methods to obtain decompositions of this type have been 
developed by other researchers~\citelist{
    \cite{chatroom} \cite{CL:blocksnoise} \cite{CL:blockalgebra }\cite{CS:blocks}}.
These patterns can be detected by Algorithm~\ref{algo:detailed-stratify} 
using an \textit{adjoint chisel}.

There is one adjoint chisel for each (unordered) pair of modes that define 
the face in question. Thus, for a tensor with $\ell$ modes, there are 
$\ell \choose 2$ adjoint chisels, each consisting of a single row. 
The adjoint chisel for Figure~\ref{fig:main-block}{\small {\sc (B)}} 
is 
\begin{align}
\label{eq:adjoint_chisel}
\ch_{{\rm adj}}=\left[
        \begin{array}{ccc}
            1 & -1 & 0 
        \end{array}
    \right].
\end{align}
The solution space $\Der(\ch_{\mathrm{adj}},\Gamma)$ is an associative 
algebra whose product is coordinatewise composition~\cite{BW-tensor}.
The subspace of scalar $\ch_{{\rm adj}}$-derivations is given by 
\begin{align} 
    \label{eq:scal-adj}
    \mathrm{scl}(\ch_{\mathrm{adj}}) &= 
        \{~(\lambda I_{d_1},\lambda I_{d_2},0_{d_{3}})~:~\lambda\in\K\,\},
    \end{align}
and is therefore 1-dimensional.
Hence, when using $\ch_{\mathrm{adj}}$ 
in Algorithm~\ref{algo:detailed-stratify}, we focus on the second 
smallest singular value of the matrix $N$ in Line 3.

As one of the modes 
is disengaged in the adjoint 3-chisels, 
it is not surprising that 
there is no restriction on the data in the tensor along that mode.
In fact, the adjoint chisels reveal `face block' structure 
along the face corresponding to the two active modes. 
Figure~\ref{fig:adj-chisel} shows the outcome of the now 
familiar experiment.

\begin{figure}[!htbp]
    \centering
        \includegraphics{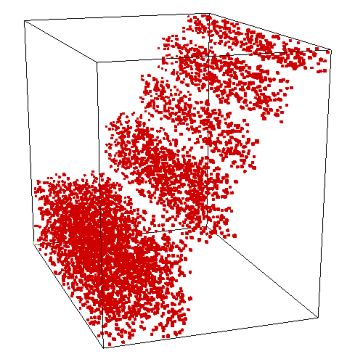}
        \hspace{2cm}
        \includegraphics{images/tensor-face-recov-print.png}
    \caption{{\small The $(50\times 60\times 70)$-array 
    on the right was reconstructed by our algorithm 
    using the $\ch_{{\rm adj}}$ chisel 
    after we scrambled the synthetically constructed tensor on the left.
    }}
    \label{fig:adj-chisel}
\end{figure}

  
\subsection{The centroid chisel}
\label{sec:centroid}
If we discover face block decompositions in two faces 
of a 3-tensor simultaneously we get diagonal blocks. In particular, 
we can determine 
whether a tensor decomposes into diagonal blocks using 
the single chisel 
\begin{align}
    \label{eq:cen_chisel}
\ch_{\mathrm{cen}} ~& =~ \begin{bmatrix} 
    1 & -1 & 0 \\ 0 & 1 & -1
\end{bmatrix}.
\end{align}
For a given tensor $\Gamma$, the space 
$\Cen(\Gamma):=\Der(\ch_{{\rm cen}},\Gamma)$
has the structure of a commutative algebra, called the 
\textit{centroid} of $\Gamma$, whose product is 
again coordinatewise composition.
This ring and its properties have been studied in numerous ways
\citelist{\cite{JLP-Border-Rank}\cite{chatroom2}\cite{Genus2}\cite{Wilson:direct-decomp}}.
Figure~\ref{fig:cen-chisel} illustrates an
experiment to recover diagonal blocks.

\begin{figure}[!htbp]
    \centering
    \includegraphics{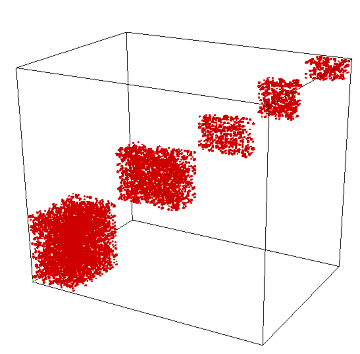}
    \hspace{2cm}
    \includegraphics{images/tensor-blocks-diag-recov-print.png}
    \caption{{\small The array with prescribed block structure on the 
    left was scrambled and then passed as input to 
    Algorithm~\ref{algo:detailed-stratify}
    with the $\ch_{{\rm cen}}$ chisel; the array on the right is the 
    output.}}
    \label{fig:cen-chisel}
\end{figure}


\subsection{Explaining the outcomes}
\label{sec:reasons}
Why do the chisels we have highlighted in this section 
reveal the sparsity patterns we have seen in these experiments?
Let us start with the adjoint chisel $\ch_{\mathrm{adj}}=[0,1,-1]$ and 
a $\ch_{\mathrm{adj}}$-derivation $(D_1, D_2, D_3)$, so that
\begin{align*}
0 \cdot \langle \Gamma \mid D_1 u_1, u_2, u_3\rangle + 1\cdot \langle \Gamma \mid u_1,D_2 u_2, u_3\rangle+(-1)\cdot \langle \Gamma\mid u_1, u_2, D_3 u_3\rangle &= 0.
\end{align*}
The first mode is not engaged so we insist that $D_1=0$ (Remark~\ref{rmk:engaged}).
We can use partial evaluation to remove mode 1 from consideration by setting
\begin{align*} 
u * v & := \langle \Gamma | -, u, v\rangle.
\end{align*}
Putting $X:=D_3$ and $X^{\circ}:=D_2$, the $\ch_{\mathrm{adj}}$-derivation condition for
$(0,X^{\circ},X)$ then becomes the familiar adjoint condition 
\begin{align*} 
(X^{\circ} u)*v  &= u*(X v)
\end{align*}
from linear algebra expressed for the bilinear map $*:U_2\times U_3\to U_1^*$ 
and this means
\begin{align}
    \label{eq:orth}
(\text{Null } X^{\circ})*(Xv)=X^{\circ} (\text{Null } X^{\circ})*v=0.
\end{align}
Because $(0,\lambda I,\lambda I)$ is a $\ch_{\mathrm{adj}}$-derivation, 
it follows that $(0,X^{\circ}-\lambda I, X-\lambda I)$ is also a 
$\ch_{\mathrm{adj}}$-derivation. 
Thus, to find and diagonalize a $[0,1,-1]$-derivation of $\Gamma$ is 
to find a sparsity pattern on the $\{2,3\}$-face of $\Gamma$.  
\smallskip

Next consider the centroid chisel $\ch_{\mathrm{cen}}$, and note 
that the space of $\ch_{\mathrm{cen}}$-derivations is  
the intersection of adjoint derivations on two or more faces of $\Gamma$.  
Thus the sparsity patterns for $\ch_{\mathrm{cen}}$ coincide with blocks 
along the diagonal of the array.
\smallskip 

The surfaces revealed by the universal chisel $\ch_{\mathrm{uni}}$ 
are the most surprising, but we have already seen in the discussion about the 
sphere in Example~\ref{ex:sphere-lab} how these continuous cases are 
indeed sparsity patterns for the universal chisel.

\subsection{Tucker and Higher-Order Singular Value Decompositions}
\label{sec:tucker}
We mentioned in the introduction that some known
tensor decompositions are special cases of chiseling. We  
explain how these decompositions are recovered by our methods.
\smallskip

Tensor decompositions of the form depicted in Figure~\ref{fig:structure-1}{\small (A)}  
are known as Tucker decompositions~\cite{Tucker}, and many efficient algorithms have 
been developed to find them~\cite{Bader-Kolda}.  
Proposition~\ref{prop:classify_chisels} describes the chisels 
needed to carry out Tucker decompositions on 3-tensors.
In particular, a full Tucker decomposition may be carried 
out using the identity matrix as a chisel, since
condition ~\eqref{eq:C-derivation} then
becomes
\begin{align*}
    \forall a\in[\ell], &&
\langle \Gamma | u_1,\ldots, D_a u_a, \ldots, u_{\ell}\rangle &=0.
\end{align*}
In words, the image of $D_a$ lies in the radical of $\langle \Gamma\mid-\rangle$,  
and the $a$-th Tucker decomposition is precisely that of the image and 
kernel of $D_a$. A demonstration of Tucker decomposition via 
chiseling can be found in the notebook~\cite{OpenDleto}*{Chiseling101}.
\smallskip

Focusing attention on a single axis $a$ of a multilinear 
map $\gamma:U_1\times \cdots \times U_{\ell}\to \K$ is functionally the same 
as treating $\gamma$ as a linear map 
\begin{align*}
L_a &: U_a\to (U_1\otimes \cdots U_{a-1}\otimes 
U_{a+1}\otimes \cdots \otimes U_{\ell})^*.
\end{align*}  
The nullspace of $L_a$ is the $a$-th radical and the derivation condition 
$L_a D_a=0$ identifies this nullspace.  More generally, the singular values 
of $L_a$ are known as higher-order singular values. Our methods find $D_a$  
such that $L_a D_a\approx 0$, so the $D_a$ eigenspaces coincide with the 
lowest HOSVD components of the tensor.
\smallskip

\textit{We reiterate that we are not proposing that our algorithm 
is a replacement for the highly optimized Tucker and HOSVD 
decomposition algorithms}, 
only that they can be viewed as special cases of chiseling.

\section{The theory of chisels}
\label{sec:theory}
The previous section focused on a small collection of chisels 
designed to reveal specific types of sparsity patterns.
These are also the chisels we have found most useful in our experience.
However, there is more to say about chiseling in general, and
our methods have the capability to work with chisels 
in more refined ways than the basic algorithms describe. 
In this section we elaborate on tensor chiseling.

\subsection{Ordering chisels and derivations}
\label{sec:orders}
For chisels $\ch$ and $\ch'$, each with $\ell$ columns,  
define the relation $\ch\leq \ch'$ if every row of $\ch$ is a linear combination 
of the rows of $\ch'$.  Write $\ch\equiv \ch'$ as chisels if 
$\ch\leq \ch'$ and $\ch'\leq \ch$, namely if $\ch$ and $\ch'$  
have the same row span.
Given $\ch\leq \ch'$, a tensor $\Gamma$, and 
$D\in \Der(\ch',\Gamma)$, linear combinations of the 
equations in condition~\eqref{eq:C-derivation} 
can be used to show that $D\in \Der(\ch,\Gamma)$, so we have
\begin{align}
\label{eq:galois}
    \ch \leq \ch' \qquad \Rightarrow \qquad \Der(\ch,\Gamma) \geq \Der(\ch',\Gamma).
\end{align}
Thus, the derivations of a chisel $\ch$ are determined by the row span pf $\ch$.
We learn more about condition~\eqref{eq:C-derivation} 
by considering its dual. Fix a set 
\begin{align*}
    \Omega &\subseteq
    \mathbb{K}^{d_1\times d_1}\times \cdots \times \mathbb{K}^{d_{\ell}\times d_{\ell}}
\end{align*}
of transverse operators. For an $\ell$-tensor $\Gamma$, put
\begin{align*}
    L(\Gamma,\Omega) &:= \{\, \ch \in \mathbb{K}^{c\times \ell} : 
    \text{condition~\eqref{eq:C-derivation} holds for all }(D_1,\ldots,D_{\ell})\in \Omega\,\}.
\end{align*}
Evidently, 
$\Omega\subseteq \Omega'$ implies $L(\Gamma,\Omega)\supseteq L(\Gamma,\Omega')$,
so by duality we have
\begin{align*}
    \ch\subseteq L(\Gamma,\Omega) &\iff \Omega \subseteq \Der(\ch,\Gamma) 
\end{align*}
Composing $\Der(-,\Gamma)$ and $L(\Gamma,-)$ produces closure operators on
the spaces of chisels and derivations:
\begin{align*}
    \lceil \ch\rceil_{\Gamma} &:= L(\Gamma,\Der(\ch,\Gamma))
    & 
    \lceil \Omega\rceil_{\Gamma} & := \Der(L(\Gamma,\Omega),\Gamma).
\end{align*}
We say that a chisel $\ch$ is \textit{closed} with respect to a tensor 
$\Gamma$ if the rows of $\ch$ span $\lceil \ch\rceil_{\Gamma}$. 
A fuller account of such orders and closures is given in 
\citelist{\cite{BW-tensor}\cite{FMW}}.

\subsection{Classifying 3-chisels}
\label{sec:classify}
Chisels with equal row spaces determine the same spaces of derivations, 
but we argue that a weaker equivalence on chisels is appropriate for the 
purpose of classification. Multiplying a chisel $\ch$ by an invertible 
diagonal matrix $\Sigma=\text{Diag}(\sigma_1,\ldots,\sigma_{\ell})$ rescales 
the columns of $\ch$, and 
\begin{align*} 
(D_1,\ldots,D_{\ell}) &\in\Der(\ch,\Gamma) \qquad \iff \qquad 
(\sigma_1 D_1,\ldots,\sigma_{\ell} D_{\ell}) \in \Der(\ch\Sigma,\Gamma).
\end{align*}
If a $\ch$-derivation $(D_1,\ldots,D_{\ell})$ detects a pattern 
$\Delta(\ch,\delta)$, then the rescaled derivation
$(\sigma_1 D_1,\ldots,\sigma_{\ell} D_{\ell})$ detects the pattern
$\Delta(\ch\Sigma,\Sigma^{-1}\delta)$.  
For the purpose of detecting patterns
it therefore suffices to work with full row rank matrices $\ch$ up to the transformation
$M\ch\Sigma$ for invertible $M$ and diagonal invertible $\Sigma$.  
Furthermore, composing $\Sigma$ with a permutation matrix only 
reindexes the modes of the chisel, which further simplifies their classification.
\smallskip

For 2-tensors (linear maps) 
the nonzero chisels are $[1~1]$, $[1~0]$, and 
$\left[ \begin{smallmatrix} 1 & 0 \\ 0 & 1 \end{smallmatrix} \right]$  
up to this equivalence.  
The next result gives the classification of chisels for 3-tensors.

\begin{prop} 
    \label{prop:classify_chisels}
    A nonzero 3-chisel is equivalent to exactly one of the following:
\begin{align*}
    \text{\small (Third Tucker chisel)}\quad
    \begin{bmatrix}
        1 & 0 & 0 \\
        0 & 1 & 0 \\
        0 & 0 & 1 
    \end{bmatrix} & &
    \text{\small (Centroid chisel)}\quad
    \begin{bmatrix}
        1 & 0 & 1\\
        0 & 1 & 1
    \end{bmatrix}\\
    \text{\small (Second Tucker chisel)}\quad
    \begin{bmatrix}
        0 & 1 & 0 \\ 
        0 & 0 & 1 
    \end{bmatrix} &&
    \text{(Adjoint chisel)}\quad
    \begin{bmatrix}
        0 & 1 & 1
    \end{bmatrix}\\ 
    \text{\small (First Tucker chisel)}\quad
    \begin{bmatrix}
        0 & 0 & 1 
    \end{bmatrix}
    & &
    \text{\small (Universal chisel)} \quad
    \begin{bmatrix}
        1 & 1 & 1 
    \end{bmatrix}
\end{align*}
\begin{gather*}
    \text{\small (First Tucker + Adjoint chisel)}\quad
    \begin{bmatrix}
        1 & 0 & 0 \\
        0 & 1 & 1
    \end{bmatrix}
\end{gather*}
\end{prop}
\begin{proof}
First, compute the reduced row echelon form of $\ch$ and remove all zero rows.
Next, permute the columns so that any zero columns appear to the left. The 
resulting matrices not explicitly given in the classification have the form 
\begin{align*}
\begin{bmatrix}
1 & 0 & a\\ 0 & 1 & b 
\end{bmatrix}.
\end{align*}
If $ab=0$, this matrix is equivalent to 
$\left[\begin{smallmatrix} 1 & 0 & 0 \\ 0 & 1 & 1 \end{smallmatrix}\right]$.
If $ab\neq 0$, the matrix is equivalent to
    \begin{align*}
        \begin{bmatrix}
            1/a & 0 \\ 0 & 1/b 
        \end{bmatrix}
        \begin{bmatrix} 
            1 & 0 & a \\ 
            0 & 1 & b 
        \end{bmatrix}
        \begin{bmatrix}
            a & 0 & 0 \\ 0 & b & 0 \\ 0 & 0 & 1
        \end{bmatrix}
        =
        \begin{bmatrix}
            1 & 0 & 1 \\ 
            0 & 1 & 1
        \end{bmatrix}.
        \qquad 
    \end{align*}
    This completes the classification.
\end{proof}

\begin{ex}
Consider the 3-tensor 
$\mathbb{K}^{a\times b}\times \mathbb{K}^{b\times c}\to \mathbb{K}^{a\times c}$
determined by rectangular matrix multiplication.
By direct calculation, the space of derivations for the centroid chisel 
is 1-dimensional (just the field of definition). 
The derivations of the adjoint chisels $[1~1~0]$, $[1~0~1]$, 
and $[0~1~1]$ are, respectively, $\mathbb{K}^{a\times a}$, $\mathbb{K}^{b\times b}$, 
and $\mathbb{K}^{c\times c}$.  The space of
derivations of the universal chisel $[1~1~1]$ 
is isomorphic to $\mathbb{K}^{a^2+b^2+c^2-1}$, and contains all of the others.  
We remark that the centroid, adjoint, and universal chisels are, in the sense of 
Section~\ref{sec:orders}, 
all closed for this tensor.  
\end{ex}

\subsection{Higher valence chisels}  
It suffices, when working with tensors of valence $\ell$, 
to consider just the reduced row echelon
forms of the $m\times \ell$ chisel matrices $\ch$: 
\begin{align*}
\ch &= \begin{bmatrix} I_r & \ch' \\ 0 & 0\end{bmatrix}.
\end{align*}
Hence, there are at most $r\times (\ell-r)$ parameters to describe 
chisels of rank $r$, corresponding to the different choices of matrix $\ch'$.  
Furthermore, if $\Sigma$ is an invertible $r\times r$ diagonal matrix, 
and $\Sigma'$ a diagonal $(\ell-r)\times (\ell-r)$ matrix, 
then 
\begin{align*} 
\begin{bmatrix}\Sigma^{-1} & 0 \\ 0 & I\end{bmatrix} 
\begin{bmatrix} I_r & \ch'\\ 0 & 0 \end{bmatrix} 
\begin{bmatrix} \Sigma & 0 \\ 0 & \Sigma' \end{bmatrix}
&=\begin{bmatrix} I_r & \Sigma^{-1}\ch' \Sigma' \\ 0 & 0 \end{bmatrix}.
\end{align*}
Hence, up to our equivalence, there are $(r-1)(\ell-r-1)$ 
free parameters in $\ch'$,  
which is positive when $1<r<\ell-1$.
Thus, when $\ell\geq 4$, the number of inequivalent chisels is at 
least the size of the field $\K$, which may be infinite or uncountable.

For instance, if we define, for $a\in \K\setminus\{0\}$, the 
family of 4-chisels
\begin{align*}
    \ch(a) & = \begin{bmatrix} 1 & 0 & 1 & a \\ 0 & 1 & 1 & 1 \end{bmatrix},
\end{align*}
then $\ch(a)$ is equivalent to $\ch(b)$ if, and only if, $a=b$. Hence, over infinite fields 
there are infinitely many 
non-equivalent chisels of valence $\ell>3$.  

\subsection{Variations}
\label{sec:advanced}
One may wish to impose restrictions on the $\ch$-derivations 
one uses to explore the structure of a tensor $\Gamma$. 
The tensor $\Gamma$ may possess internal symmetries 
that impose constraints among the transverse operators 
$(D_1,\ldots,D_{\ell})$, or it may simply be that our application 
requires that we only work with symmetric or Hermitian operators.
This can be achieved by specifying a suitable 
subspace $\Omega\leq \K^{d_1\times d_1}\times \cdots \times \K^{d_{\ell}\times d_{\ell}}$
of transverse operators and then considering the subspace 
\begin{align*}
    \Der_{\Omega}(\ch,\Gamma) &:= \Der(\ch,\Gamma)\cap \Omega.
\end{align*}
of derivations that lie within $\Omega$. Since 
the eigenspaces of a symmetric matrix are orthogonal,
imposing the constraint that $D_a$ is symmetric 
produces orthogonal decompositions along the $a^{\mathrm{th}}$ 
mode of $\Gamma$.




\subsection{The algebraic structure \boldmath$\Der_{\Omega}(\ch,\Gamma)$}
\label{sec:algebra-structure}
Singular and eigenvalue theories possess uniqueness results that 
ensure reproducibility. On first inspection chiseling may appear to 
lack any such uniqueness guarantees. For instance, it may concern the 
reader that we are happy to use  
any (nonscalar) derivation $\delta$ in $\Der(\ch,\Gamma)$ we can get our hands on.  
However, the sets $\Der(\ch,\Gamma)$ generally possess an algebraic structure 
that explains why generic choices produce the same sparsity patterns. 

For example, if a chisel $\ch$ consists entirely of $0$'s and $1$'s, 
then $\Der(\ch,\Gamma)$ is closed under
the Lie algebra product $[D,E]:=(D_a E_a - E_a D_a : 0\leq a\leq \ell)$.  
The same is true of an arbitrary $\ch$ in
projective coordinates (where $u_a$ in \eqref{eq:C-derivation} is replaced by $\K u_a$). 
Derivations for $(\Omega,\ch)$ pairs can admit  algebraic structures, such as 
associative $*$-algebras \cite{BW-tensor}, Jordan algebras
\citelist{\cite{MM:star-alge-blocks} \cite{Wilson:unique} \cite{Wilson:central}}, or 
commutative rings  \citelist{\cite{chatroom2} \cite{Genus2} \cite{JLP-Border-Rank}
\cite{Wilson:direct-decomp}}. Conditions on the possible algebraic products are 
provided in~\cite[Theorem D]{FMW}.

\section{Universality \& uniqueness}
\label{sec:uniqueness}
We have alluded throughout to uniqueness properties 
of the tensor decompositions we discover. In this section 
we prove such uniqueness claims for the chisel 
\begin{align}
    \label{eq:universal-chisel}
    \ch_{{\rm uni}} &= \left[
        \begin{array}{cccc}
            1 & 1 & \cdots & 1 
        \end{array}
    \right].
\end{align}
The proof uses the fact that $\Der(\ch_{{\rm uni}},\Gamma)$ has the structure of a Lie algebra.
Similar uniqueness results hold for other chisels $\ch$ for which $\Der(\ch,\Gamma)$
has a nice algebraic structure---the chisels in 
Sections~\ref{sec:universal}---\ref{sec:centroid} 
are of this type (cf.~Section~\ref{sec:algebra-structure}).

\subsection{The universal chisel}
We begin by explaining why we use $\ch_{\mathrm{uni}}$ 
as the default chisel, and why we refer to it as `universal'.
The following result, a special case of \cite{FMW}*{Theorem~1}, 
asserts that $\ch_{{\rm uni}}$ is the least restrictive chisel: if $\Gamma$ has 
only scalar $\ch_{{\rm uni}}$-derivations, then it has only scalar 
$\ch$-derivations for \textit{any} chisel $\ch$. To coin a phrase, 
there is one chisel to rule them all.

\begin{prop}
\label{prop:universal-chisel}
Let $\Gamma$ be a tensor defined over an infinite field $\K$. 
For any chisel $\ch$  
and any $\ch$-derivation $(D_1,\ldots,D_{\ell})$ of $\Gamma$,
there are nonzero scalars $(\lambda_1,\ldots,\lambda_{\ell})$ 
such that $(\lambda_1 D_1,\ldots,\lambda_{\ell}D_{\ell})$ 
is a $\ch_{{\rm uni}}$-derivation of $\Gamma$. 
\end{prop}

\begin{proof}
    Recall from Remark~\ref{rmk:engaged} the convention that $D_b=0$ whenever column $b$ of $\ch$ 
    is zero, so we may focus attention on the subset $A\subseteq [\ell]$ 
    upon which $\ch$ is engaged. 
    For $v\in\mathbb{K}^m$ and $a\in A$
    define $\tau_{av}:= (v\ch)_a=\sum_j v(j) \ch_{ia}$.
    For each $a\in [\ell]$, there exists $j$ with 
    $\ch_{ja}\neq 0$. 
    Hence, $0=\tau_{av}=\sum_j v(j)\ch_{ja}$ is a nonzero homogeneous 
    system of linear equations and its solution space is a hyperplane.  
    Observe that $\tau_{av}\neq 0$ for some $a\in[\ell]$. (Otherwise,
    $\K^m$ would be a union of finitely many hyperplanes, 
    which is impossible since $\K$ is infinite.)
    Hence, there exists $v\in \K^m$ such that 
    $(\tau_{av} : a\in [\ell])$ has no zero entry.  
    Since $D$ is a $\ch$-derivation and therefore also a $v\ch$-derviation, 
    it follows that $(\tau_{av}^{-1} D_a :a\in [\ell])$ is a 
    $\ch_{{\rm uni}}$-derivation.
\end{proof}

\subsection{Uniqueness}
The following result provides the theoretical 
foundation for our uniqueness claims for the chisel $\ch_{{\rm uni}}$.  

\begin{thm}
\label{thm:unique}
    For a tensor $\Gamma$ defined over $\mathbb{C}$, the following hold:
    \begin{enumerate}
        \item[(a)] Let $(D_a : a\in [\ell])$ be a 
        $\ch_{{\rm uni}}$-derivation of $\Gamma$. For each $a\in[\ell]$, 
        the operator $D_a$ can be decomposed as $D_a=S_a+N_a$, where  
        $S_a$ is diagonalizable and $N_a$ is nilpotent, 
        $S_a$ and $N_a$ commute,
        and any two such decompositions
        are conjugate under an invertible operator on $U_a$.
        Furthermore, both 
        \begin{align*}
        (S_a : a\in [\ell]) & \qquad \text{and} \qquad 
        (N_a : a\in [\ell]) 
        \end{align*}
        are $\ch_{{\rm uni}}$-derivations of $\Gamma$.

        \item[(b)] Let $(D_a : a\in [\ell])$ 
        and $(E_a : a\in [\ell])$ be $\ch_{{\rm uni}}$-derivations 
        of $\Gamma$ such that all $D_a$ and $E_a$ are diagonalizable.  
        For each $a\in [\ell]$, there are invertible matrices 
        $X_a$ such that 
        $(X_a^{-1}E_aX_a : a\in [\ell])$ is a $\ch_{{\rm uni}}$-derivation
        of $\Gamma$, and
        $D_a$ and $X_a^{-1} E_a X_a$ are 
        simultaneously diagonalizable.  
    \end{enumerate}
\end{thm}

\begin{proof}
    Given $\ch_{{\rm uni}}$-derivations 
    $(D_a : a\in [\ell])$ and $(E_a : a\in [\ell])$ 
    and $\lambda\in\mathbb{C}$,  
    \begin{align*}
        (D_a+\lambda E_e : a\in [\ell]) && \mbox{and}&& (D_a E_a-E_aD_a : a\in [\ell])
    \end{align*}
    are both $\ch_{{\rm uni}}$-derivations. Thus, the set $\mathrm{Der}(\ch_{{\rm uni}},\Gamma)$ of all 
    $\ch$-derivations of $\Gamma$ is a $\C$-Lie algebra, and (a)
    follows from standard theory~\cite[Chapter III, Theorem 16]{Jac}.

    Each $\C$-Lie algebra has a maximal (\textit{Cartan}) subalgebra of 
    simultaneously diagonalizable elements.  
    Every diagonalizable element lies in a Cartan subalgebra, 
    and  two Cartan subalgebras are conjugate~\cite[Chapter IX, Theorem 3]{Jac}.
    In (b), for each $a\in[\ell]$, let $H_a$   
    be a Cartan subalgebra containing
    $D_a$, and let $J_a$ be a Cartan subalgebra containing $E_a$.
    Let $X_a$ be an invertible matrix conjugating $J_a$ to $H_a$. Then 
    $D_a$ and $X_a^{-1}E_aX_a$ lie in a common Cartan subalgebra, so they 
    are simultaneously diagonalizable. 
\end{proof}

Theorem~\ref{thm:unique} identifies the optimal sparsity pattern determined 
by $\ch_{{\rm uni}}$. In particular, 
a single random derivation usually provides all of the information we need:

\begin{coro}
    Any two $\ch_{{\rm uni}}$-derivations of a tensor $\Gamma$
    define a unique joint sparsity pattern.
\end{coro}
\begin{proof}
    Given two $\ch_{{\rm uni}}$-derivations, select their semisimple components 
    from Theorem~\ref{thm:unique}(a). Next, use Theorem~\ref{thm:unique}(b) to 
    conjugate those semisimple parts in a common Cartan subalgebra.
    They now have common eigenspaces, and hence define 
    a common sparsity pattern. 
\end{proof}

Over $\mathbb{R}$ the situation is the same except that we occasionally need
$2\times 2$ blocks with complex eigenvalues, so 
the resulting sparsity patterns may be larger.

\subsection{Ordering the blocks}
If $\{X_i : i\in [k]\}$ is the set of eigenspaces of a 
complex linear operator, the corresponding eigenvalues are 
returned with an arbitrary ordering. This means  
that every sparsity pattern we produce is only unique up to reording.     
One often has to contend with such issues when seeking sparsity patterns, 
but our algorithms can produce so many blocks that it 
behooves us to exert some control when we can.

\begin{prop}
\label{prop:unique-order}
    If the space of $\ch$-derivations of a tensor has dimension 
    exactly one more than 
    that of the null space of $\ch$, and if there is a non-scalar $\ch$-derivation 
    $D$ with all real eigenvalues, then 
    there is a well-defined unique canonical 
    ordering of the associated sparsity pattern.
\end{prop}

\begin{proof}
    If $D=(D_a : a\in [\ell])$ is such a $\ch$-derivation, then
    each $D_a$ can be taken to be 
    a diagonal matrix with ordered eigenvalues 
    $\delta_{a}(1),\ldots, \delta_{a}(k_a)$.
    As the space of $\ch$-derivations is one-dimensional over the scalar 
    derivations, it follows that any non-scalar derivation 
    $E_a$ has eigenvalues $\{\delta_{a}(1)+\beta_a,
    \ldots, \delta_{a}(k_a)+\beta_a\}$ for some scalar $\beta_a$.
    In particular, the ordering of the eigenspaces inherited from 
    eigenvalues is invariant under every choice of non-scalar derivation.
\end{proof}

The conditions of Proposition~\ref{prop:unique-order} are usually met 
for a sufficiently generic tensor possessing a nontrivial sparsity pattern $\Delta(\ch,\delta)$. 
Even if they are not,
however, one can appeal to more subtle ordering tools from Lie theory. 
Avoiding a technical digression, one idea is to order the 
eigenspaces along a composition series and then appeal to highest weight 
vectors. This approach would handle all but the abelian Lie algebras 
whose dimension over the scalars is greater than 1; no canonical order is known for 
such Lie algebras.

\section{Summary}
\label{sec:conclusion}
We have introduced the concept of \emph{sparsity patterns}, which capture the outputs
of many of the tensor clustering algorithms in current use. We also introduced 
\textit{chisels} to parameterize classes of computable sparsity patterns, and showed 
that these classes include new patterns not known to be detectable by existing methods. 
We proved that, over infinite fields, one particular chisel is universal among all chisels.
We also described an efficient algorithm to detect the existence of sparsity patterns in 
tensors, and proved that the resulting patterns are unique.

\subsection*{Acknowledgments}

This research began as a challenge from colleagues in applied mathematics to 
adapt algebraic algorithms designed for fields with exact arithmetic 
to produce practicable tools for computing with tensors over fields with approximate 
operations. While these colleagues remain unnamed, we thank them both for their 
critiques and for their encouragement.
We also thank the Departments of Mathematics at  Bucknell \& Cornell University for jointly 
hosting our team meetings that led to the discovery of the original chisel method.  
We thank Nick Vannieuwenhoven for engaging discussions on optimizing chisel algorithms 
in September 2024. Finally, we thank the anonymous referees for their careful reading 
and helpful suggestions that led to significant improvements in the paper.

The work 
was initially supported by the Simons Foundation.
P.A. Brooksbank was supported by 
National Science Foundation Grant DMS 2319372.
M.D. Kassabov was supported by 
National Science Foundation Grant DMS 2319371.
J.B. Wilson was supported by 
National Science Foundation Grant DMS 2319370.

\bibliographystyle{abbrv}

\begin{bibdiv}
\begin{biblist}

\bib{chatroom}{article}{
    author = {Acar, Evrim},
    author={\c{C}amtepe, Seyit A.},
    author={Krishnamoorthy, M.S.},
    author={Yener, B\"{u}lent},
    year={2005},
    title={Modeling and Multiway Analysis of Chatroom Tensors},
    booktitle={Intelligence and Security Informatics. ISI 2005. Lecture Notes in Computer Science},
    volume={3495},
    publisher={Springer},
    address={Berlin, Heidelberg},
}

\bib{chatroom2}{article}{
    author = {Acar, Evrim},
    author={\c{C}amtepe, Seyit A.},
    author={Yener, B\"{u}lent},
    title = {Collective Sampling and Analysis of High Order Tensors for Chatroom Communications},
    booktitle = {Proceedings of the 4th IEEE International Conference on Intelligence and Security Informatics},
    series = {ISI'06},
    year = {2006},
    pages = {213--224},
    publisher = {Springer-Verlag},
    address = {Berlin, Heidelberg},
}

\bib{Julia}{article}{
  title={Julia: A fresh approach to numerical computing},
  author={Bezanson, Jeff and Edelman, Alan and Karpinski, Stefan and Shah, Viral B},
  journal={SIAM review},
  volume={59},
  number={1},
  pages={65--98},
  year={2017},
  publisher={SIAM},
  url={https://doi.org/10.1137/141000671}
}

\bib{BC:hypergraphs}{article}{
   author={Brooksbank, Peter A.},
   author={Chaplin, Clara R.},
   title={New linear invariants of hypergraphs},
   note={arxiv:2512.03342}
   year={2025}
}

\bib{BW-tensor}{article}{
   author={Brooksbank, Peter A.},
   author={Wilson, James B.},
   title={Groups acting on tensor products},
   journal={J. Pure Appl. Algebra},
   volume={218},
   date={2014},
   number={3},
   pages={405--416},
   issn={0022-4049},
   review={\MR{3124207}},
   doi={10.1016/j.jpaa.2013.06.011},
}
     
\bib{Genus2}{misc}{
   author={Brooksbank, Peter A.},
   author={Maglione, Joshua},
   author={Wilson, James B.},
   title={A package for isomorphism testing of groups and tensors of tame genus},
   url={https://github.com/thetensor-space/TameGenus},
   year={2018},
   version={2.0},
}

\bib{BMW:Rihanna}{article}{
    author={Brooksbank, Peter A.},
    author={Maglione, Joshua},
    author={Wilson, James B.},
    title={Exact sequences of inner automorphisms of tensors},
    journal={J. Algebra},
    volume={545},
    date={2020},
    pages={43--63},
    issn={0021-8693},
    review={\MR{4044688}},
    doi={10.1016/j.jalgebra.2019.07.006},
}

\bib{OpenDleto}{webpage}{
    author={Brooksbank, P. A.},
    author={Kassabov, M. D.},
    author={Wilson, J. B.},
    title={OpenDleto},
    subtitle={Algorithms for structure recovery in tensors},
    url={https://github.com/thetensor-space/OpenDleto},
    date={}
}


\bib{CL:blockalgebra}{article}{
    author={Cai, Yunfeng},
    author={Liu, Chengyu},
    title={An algebraic approach to nonorthogonal general joint block diagonalization},
    journal={SIAM J. Matrix Anal. Appl.},
    volume={38},
    date={2017},
    number={1},
    pages={50--71},
    issn={0895-4798},
    review={\MR{3592077}},
    doi={10.1137/16M1080756},
}

\bib{CL:blocksnoise}{article}{
    author={Cai, Yunfeng},
    author={Li, Ren-Cang},
    title={Perturbation analysis for matrix joint block diagonalization},
    journal={Linear Algebra Appl.},
    volume={581},
    date={2019},
    pages={163--197},
    issn={0024-3795},
    review={\MR{3983068}},
    doi={10.1016/j.laa.2019.07.012},
}
    
\bib{CS:blocks}{article}{
    author={Cai, Yunfeng},
    author={Shi, Decai},
    author={Xu, Shufang},
    title={A matrix polynomial spectral approach for general joint block diagonalization},
    journal={SIAM J. Matrix Anal. Appl.},
    volume={36},
    date={2015},
    number={2},
    pages={839--863},
    issn={0895-4798},
    review={\MR{3357633}},
    doi={10.1137/130935264},
}    

\bib{Cardoso}{article}{
    author = {Cardoso, J.-F},
    year = {1991},  
    month = {05},
    pages = {3109 - 3112 vol.5},
    title = {Super-symmetric decomposition of the fourth-order cumulant tensor. Blind identification of more sources than sensors},
    volume = {5},
    isbn = {0-7803-0003-3},
    doi = {10.1109/ICASSP.1991.150113}
}

\bib{DeL:Blocks}{article}{
    author = {De Lathauwer, Lieven},
    title = {Decompositions of a Higher-Order Tensor in Block Terms—Part II: Definitions and Uniqueness},
    journal = {SIAM Journal on Matrix Analysis and Applications},
    volume = {30},
    number = {3},
    pages = {1033-1066},
    year = {2008},
    doi = {10.1137/070690729},
}

\bib{DeL:HOSVD}{article}{
    author = {De Lathauwer, Lieven},
    author = {De Moor, Bart},
    author = {Vandewalle, Joos},
    title = {A Multilinear Singular Value Decomposition},
    journal = {SIAM Journal on Matrix Analysis and Applications},
    volume = {21},
    number = {4},
    pages = {1253-1278},
    year = {2000},
    doi = {10.1137/S0895479896305696},
}

\bib{Eberly-Giesbrecht}{article}{
   author={Eberly, W.},
   author={Giesbrecht, M.},
   title={Efficient decomposition of associative algebras over finite
   fields},
   journal={J. Symbolic Comput.},
   volume={29},
   date={2000},
   number={3},
   pages={441--458},
   issn={0747-7171},
   review={\MR{1751390}},
   doi={10.1006/jsco.1999.0308},
}

\bib{FMW}{article}{
    title={A spectral theory for transverse tensor operators},
    author={First, Uriya},
    author={Maglione, Joshua},
    author={Wilson, James B.},
    note={arXiv:1911.02518}
    year={2020},
}

\bib{ITensor}{article}{
	title={The ITensor Software Library for Tensor Network Calculations},
	author={Fishman, Matthew}, 
    author={White, Steven R.},
    author={Stoudenmire, E. Miles},
	journal={SciPost Phys. Codebases},
	pages={4},
	year={2022},
	publisher={SciPost},
	doi={10.21468/SciPostPhysCodeb.4},
	url={https://scipost.org/10.21468/SciPostPhysCodeb.4}
}
     
\bib{Jac}{book}{
   author={Jacobson, Nathan},
   title={Lie algebras},
   note={Republication of the 1962 original},
   publisher={Dover Publications, Inc., New York},
   date={1979},
   pages={ix+331},
   isbn={0-486-63832-4},
   review={\MR{0559927}},
}

\bib{Jin-symmetric}{article}{
   author={Jin, Ruhui},
   author={Kileel, Joe},
   author={Kolda, Tamara G.},
   author={Ward, Rachel},
   title={Scalable symmetric tucker tensor decomposition},
   journal={SIAM J. Matrix Anal. Appl.},
   volume={45},
   date={2024},
   number={4},
   pages={1746--1781},
   issn={0895-4798},
   review={\MR{4804192}},
   doi={10.1137/23M1582928},
}

\bib{JLP-Border-Rank}{article}{
   author={Jelisiejew, Joachim},
   author={Landsberg, J. M.},
   author={Pal, Arpan},
   title={Concise tensors of minimal border rank},
   journal={Math. Ann.},
   volume={388},
   date={2024},
   number={3},
   pages={2473--2517},
   issn={0025-5831},
   review={\MR{4705743}},
   doi={10.1007/s00208-023-02569-y},
}

\bib{Bader-Kolda}{article}{
   author={Kolda, Tamara G.},
   author={Bader, Brett W.},
   title={Tensor decompositions and applications},
   journal={SIAM Rev.},
   volume={51},
   date={2009},
   number={3},
   pages={455--500},
   issn={0036-1445},
   review={\MR{2535056}},
   doi={10.1137/07070111X},
}

\bib{Lim-tensors}{article}{
   author={Lim, Lek-Heng},
   title={Tensors in computations},
   journal={Acta Numer.},
   volume={30},
   date={2021},
   pages={555--764},
   issn={0962-4929},
   review={\MR{4298222}},
   doi={10.1017/S0962492921000076},
}

\bib{MM:star-alge-blocks}{article}{
    author={Maehara, T.},
    author={Murota, K.},
    title={A numerical algorithm for block-diagonal decomposition of matrix $*$-algebras with general irreducible components.}
    journal={Japan J. Indust. Appl. Math.},
    number={27},
    pages={263–293},
    year={2010},
    doi={10.1007/s13160-010-0007-8}
}


\bib{Tucker}{article}{
    author={Tucker, L.R.},
    title={Some mathematical notes on three-mode factor analysis.},
    journal={Psychometrika},
    number={31},
    pages={279–311},
    year={1966},
    doi={10.1007/BF02289464},
}

\bib{Nick:Chisel}{article}{
      title={A chiseling algorithm for low-rank Grassmann decomposition of skew-symmetric tensors}, 
      author={Nick Vannieuwenhoven},
      year={2024},
      eprint={2410.14486},
      archivePrefix={arXiv},
      primaryClass={math.NA},
      url={https://arxiv.org/abs/2410.14486}, 
}

\bib{Wilson:direct-decomp}{article}{
    author={Wilson, James B.},
    title={Existence, algorithms, and asymptotics of direct product decompositions, I},
    journal={Groups Complex. Cryptol.},
    volume={4},
    date={2012},
    number={1},
    pages={33--72},
    issn={1867-1144},
    review={\MR{2921155}},
}

\bib{Wilson:unique}{article}{
   author={Wilson, James B.},
   title={Decomposing $p$-groups via Jordan algebras},
   journal={J. Algebra},
   volume={322},
   date={2009},
   number={8},
   pages={2642--2679},
   issn={0021-8693},
   review={\MR{2559855}},
}

\bib{Wilson:central}{article}{
   author={Wilson, James B.},
   title={Finding central decompositions of $p$-groups},
   journal={J. Group Theory},
   volume={12},
   date={2009},
   number={6},
   pages={813--830},
   issn={1433-5883},
   review={\MR{2582050}},
}
     
\end{biblist}
\end{bibdiv}

\end{document}